\documentclass[notitlepage,11pt,reqno]{amsart}
\usepackage{amssymb,amsmath,color,tocvsec2}
\usepackage[mathscr]{eucal}
\usepackage[breaklinks=true]{hyperref}

\usepackage{pgf,tikz}
\usetikzlibrary{arrows}
\usepackage[letterpaper]{geometry}
\allowdisplaybreaks

\newcommand{\Z}{\ensuremath{\mathbb Z}}

\newcommand{\dt}{\ensuremath{D_\mathcal{T}}}
\newcommand{\dg}{\ensuremath{D_\mathcal{G}}}
\newcommand{\dge}{\ensuremath{D_{\mathcal{G}_e}}}
\newcommand{\dgf}{\ensuremath{D_{\mathcal{G}_f}}}
\newcommand{\dgbar}{\ensuremath{\overline{D}_{\overline{\mathcal{G}}}}}
\newcommand{\at}{\ensuremath{\alpha_\mathcal{T}}}

\renewcommand{\geq}{\ensuremath{\geqslant}}
\renewcommand{\leq}{\ensuremath{\leqslant}}
\renewcommand{\subset}{\ensuremath{\subseteq}}
\newcommand{\tauin}{\boldsymbol{\tau}_{\rm in}}
\newcommand{\tauout}{\boldsymbol{\tau}_{\rm out}}

\DeclareMathOperator{\cof}{cof}
\DeclareMathOperator{\Id}{Id}
\DeclareMathOperator{\adj}{adj}

\DeclareMathOperator{\diag}{diag}

\newtheorem{theorem}{Theorem}[section]
\newtheorem{lemma}[theorem]{Lemma}
\newtheorem{prop}[theorem]{Proposition}
\newtheorem{cor}[theorem]{Corollary}

\newtheorem*{notation}{Definition-Notation}

\newtheorem{utheorem}{\textrm{\textbf{Theorem}}}

\theoremstyle{definition}
\newtheorem{definition}[theorem]{Definition}
\newtheorem{remark}[theorem]{Remark}

\numberwithin{equation}{section}

\begin{document}

\title{The additive-multiplicative distance matrix of a graph,
and a novel third invariant}

\dedicatory{To the memory of Ronald L.\ Graham, with admiration}

\author{Projesh Nath Choudhury}
\address[P.N.~Choudhury]{Department of Mathematics, Indian Institute of
Technology Gandhinagar, Palaj, Gandhinagar 382355, India}
\email{\tt projeshnc@iitgn.ac.in}

\author{Apoorva Khare}
\address[A.~Khare]{Department of Mathematics, Indian Institute of Science;
Analysis and Probability Research Group; Bangalore 560012, India}
\email{\tt khare@iisc.ac.in}

\begin{abstract}
Graham showed with Pollak [\textit{Bell Sys.\ Tech.\ J.}\ 1971] and
Hoffman--Hosoya [\textit{J.\ Graph Th.}\ 1977] that for any directed
(additively weighted) graph $G$ with strong blocks $G_e$, the determinant
$\det(D_G)$ and cofactor-sum $\cof(D_G)$ of the distance matrix $D_G$ can
be computed from these same quantities for the blocks $G_e$. These
phenomena have been generalized to trees -- and in our recent work
[\textit{Eur.\ J.\ Combin.}\ 2024] to any graph -- with multiplicative
and $q$-distance matrices. For trees, we went further and unified all
previous variants with weights in a unital commutative ring, into a
distance matrix with both additive and multiplicative edge-data.

In the present paper,
we extend the additive-multiplicative model to $D_G$ for every graph $G$,
and introduce a third, new invariant $\kappa(D_G)$. With these,
we prove general formulas in the above vein connecting $D_G$ to
$D_{G_e}$ that crucially require $\kappa(D_G)$;
and we further refine the state-of-the-art in every setting to minors of
$D_G$. (The simpler case of trees was in our recent work.)
The proofs involve a novel application (to our knowledge) of Zariski
density to this area. Thus, our results hold over an arbitrary
commutative ring, and to our knowledge subsume all previous versions.

In greater detail: first, we introduce the additive-multiplicative
distance matrix $\dg$ of an arbitrary strongly connected graph $G$, using
what we term the additive-multiplicative block-datum $\mathcal{G}$. This
subsumes the previously studied additive, multiplicative, and
$q$-distances for all graphs.

Second, we introduce an invariant $\kappa(\dg)$ that seems novel even
from Graham's original works till now; and use it to prove ``master''
Graham--Hoffman--Hosoya (GHH) identities, which express $\det(\dg),
\cof(\dg)$ in terms of the strong blocks $G_e$. We show how these imply
all previous variants.

Third, we show that $\det(\cdot), \cof(\cdot), \kappa(\cdot)$ depend only
on the block-data for not just $\dg$, but also several minors of $\dg$.
This extension was not studied in any setting to date. We show it holds
in the ``most general'' additive-multiplicative setting, hence in all
previous settings.

Finally, we also compute in closed-form the inverse of $\dg$. This again
specializes to all known variants. In particular, we recover the explicit
formula for $D_T^{-1}$ for additive-multiplicative trees in our recent
work [\textit{Eur.\ J.\ Combin.}\ 2024] (which itself specializes to a
result of Graham--Lov\'asz~[\textit{Adv. Math.}\ 1978] and answers a
question of Bapat--Lal--Pati [\textit{Linear Algebra Appl.}\ 2006] in
greater generality.) We show that not the Laplacian, but a closely
related matrix is the ``correct'' one to use in $\dg^{-1}$ -- for the
most general additive-multiplicative matrix $\dg$, of an arbitrary graph
$G$. As a sample example, we give closed-form expressions for $\det(\dg),
\cof(\dg), \kappa(\dg), \dg^{-1}$ for hypertrees.
\end{abstract}

\date{\today}



\maketitle

\setcounter{tocdepth}{2}
\tableofcontents

\begin{notation}
All graphs in this paper (see~\cite{Harary} for basics) are assumed to be
finite, simple, directed, and \emph{strongly connected}, i.e., for which
there exist directed paths between any two distinct vertices. (This is
assumed in order to be able to define a distance function between any two
nodes; for an undirected graph, all edges are understood to be
bidirected.) A \emph{cut-vertex} is one whose removal disconnects the
underlying undirected graph, and maximal subgraphs without cut-vertices
are called \emph{strong blocks}. Below, we always index by $E$ the strong
blocks of a graph $G$, not the edges (unless $G$ is a tree). A
\emph{hypertree} is a graph whose strong blocks are all cliques/complete
graphs.

Unless otherwise specified we work over an arbitrary commutative unital
ring $R$. For Zariski density arguments, we will first prove our desired
equations over the field $\mathbb{Q}(\{ a_e, m_{ij} \})$ generated by a
set of say $N$ variables, then observe that these equations in fact hold
in the subring of polynomial \emph{functions} $\Z[ \{ a_e, m_{ij} \} ]$
(using a Zariski dense subset of $\mathbb{A}_\mathbb{Q}^N$), and then
specialize to arbitrary $R$.

Given $n \geq 1$ and a set $V$ of size $n$, define ${\bf e}_j$ to be the
standard basis vector for $1 \leq j \leq n$; and
\[
{\bf e} = {\bf e}(V) = {\bf e}(n) := (1, \dots, 1)^T = \sum_{j=1}^n
{\bf e}_j \in R^n, \qquad J = J_n := {\bf e}(n) {\bf e}(n)^T, \qquad
[n] := \{ 1, \dots, n \}.
\]

Finally, given a matrix $A = (a_{ij}) \in R^{n \times n}$ with cofactors
$c_{ij} = (-1)^{i+j} \det A_{ij}$, its \emph{adjugate matrix} and
\emph{cofactor-sum} are, respectively,
\[
\adj(A) := (c_{ji})_{i,j=1}^n, \qquad \cof(A) := \sum_{i,j=1}^n c_{ij} =
{\bf e}(n)^T \adj(A) {\bf e}(n).
\]

\noindent Now one has the following basic lemma (see e.g.~\cite{CK-tree})
which is used below, possibly without reference:
\end{notation}

\begin{lemma}\label{Ldetcof}
Suppose $n \geq 1$ is an integer and $R$ is a commutative unital ring. If
$x$ is an indeterminate that commutes with $R$, and $A \in R^{n \times
n}$ is any matrix, then $\det(A + xJ_n) = \det(A) + x \cof(A)$. Moreover,
$\cof(A) = {\bf e}(n)^T \adj(A) {\bf e}(n) = \cof(A + xJ_n)$.
\end{lemma}

\section{The additive-multiplicative distance matrix of a
graph}\label{S1}

This paper contributes to the rich area of studying matrices associated
to graphs (see e.g.~\cite{Br,BrCv}). In it, we generalize and extend the
results in our previous work~\cite{CK-tree}, from trees to hypertrees
and to arbitrary graphs. In particular, we subsume previous works in the
literature that study distance matrices of graphs which are not trees. At
the same time, this work provides a deeper, more conceptual understanding
of the previously shown ``explicit formulas'' in the literature.

Classically, a distance function on a graph $G$ was defined to be one
that is additive across cut-vertices.
Distance matrices of graphs (especially trees) and their invariants have
been extensively studied in the literature, beginning with the seminal
work of Graham with his coauthors~\cite{GHH, GL, Graham-Pollak}. In
\cite{Graham-Pollak}, the authors computed $\det(D_T)$ for
$G=T$ a tree on $|E|$ edges (i.e.\ $|E|+1$ nodes), and showed this
depends only on $|E|$, not the structure of $T$. They also computed a
second invariant, the \textit{cofactor-sum} $\cof(D_T)$, and showed it
has the same independence property. More generally \cite{GHH}, the pair
$(\det(D_G), \cof(D_G))$ for an arbitrary (additively weighted) strongly
connected directed graph depend only on the pairs
$(\det(D_{G_e}), \cof(D_{G_e}))$ for the strong blocks $\{ G_e : e \in E
\}$ of $G$.

There have since been many variants of distance matrices proposed and
studied, including with additive edgeweights $a_e$, multiplicative
edgeweights $m_e$, and $q$-edgeweights $\frac{q^{\alpha_e} - 1}{q-1}$ --
and in all known cases, $\det(D_G)$ depends not on the graph structure
but only on the edge-data of the blocks.\bigskip

\noindent \textbf{The additive-multiplicative distance matrix of a graph.}
In prior recent work~\cite{CK-tree}, we unified all of these settings for
\textit{trees} into one common framework -- that of a tree with additive
and multiplicative edgeweights -- and showed that the same independence
property holds for the corresponding general form of the distance matrix.
We further showed that allowing greater freedom in the parameters leads
to dependence of the determinant on the tree structure.
Additionally, we worked with distance matrices over an arbitrary unital
commutative ring.
In the present paper, our first novel contribution is to extend this
``most general'' setting for trees, to arbitrary graphs:

\begin{definition}
Given a graph $G$ with strong blocks $\{ G_e : e \in E \}$ on $p_e$
vertices, an \textit{additive-multiplicative block-datum} $\mathcal{G}$
attached to $G$ is an $E$-tuple of pairs $\{ \mathcal{G}_e := (a_e,
D^*_{G_e}) : e \in E \}$ such that each $a_e \in R$, and each $D^*_{G_e}
\in R^{p_e \times p_e}$ is a matrix with all diagonal entries $1$.

Attached to this datum, we define two matrices for $G$ (assuming $G$ has
vertex set $[n]$):
\begin{enumerate}
\item The \textit{multiplicative distance matrix} $D_G^* :=
(m_{ij})_{i,j=1}^n \in R^{n \times n}$ is the matrix with principal
submatrices $D^*_{G_e}$ corresponding to the nodes of each strong block;
and for $i, j \in [n]$ not in the same block, if $v$ is any cut-vertex
along a directed path from $i$ to $j$, then we inductively define $m_{ij}
:= m_{iv} m_{vj}$. (Note: all diagonal entries in $D^*_G$ are still $1$.)

\item The \textit{additive-multiplicative distance matrix} $\dg :=
(d_{ij})_{i,j=1}^n \in R^{n \times n}$ has principal submatrices $\dge :=
a_e(D^*_{G_e} - J_{p_e})$ over the nodes of each strong block $G_e$; and
for $i, j \in [n]$ not in the same block, if $v$ is any cut-vertex along
a directed path $: i \to j$, then inductively define $d_{ij} := d_{iv} +
m_{iv} d_{vj}$.\footnote{Alternately, one can work with the ``dual''
definition $d_{ij} = d_{iv} m_{vj} + d_{vj}$ -- but this essentially
means that one deals with $(D^*_G)^T$ and $\dg^T$. We will not proceed
further along these lines in the present work. Also, $d_{ii} = 0\ \forall
i$.} 
(Note: $d_{ij}$ -- as also $m_{ij}$ in~(1) -- does not depend on choice
of $v$.)
\end{enumerate}
\end{definition}

\begin{remark}[Special cases]\label{Rspecial}
We briefly explain how the above framework encompasses (to our knowledge)
all known variants of distance matrices of graphs studied in the
literature over unital commutative rings. We refer the reader
to~\cite{CK-tree} for a detailed history and survey of the literature.
\begin{enumerate}
\item Our framework includes all known distance matrices of trees (i.e.\
with additive, multiplicative, and $q$ edgeweights),
because it includes as a special case the general additive-multiplicative
distance matrix for trees, introduced in our prior work~\cite{CK-tree}
over any unital commutative ring.
This is the special case of $\dg$ where each $G_e$ is a single edge $e =
\{ i, j \}$, and $E$ denotes the set of edges of the tree. In particular,
$D^*_{G_e} = \begin{pmatrix} 1 & m_{ij} \\ m_{ji} & 1 \end{pmatrix}$, so
that the datum $\mathcal{G}$ precisely equals $\{ \mathcal{G}_e := (a_e,
m_{ij}, m_{ji}) : e \in E \}$ as in~\cite{CK-tree}. As sample references
we mention~\cite{Aida,BKN,BLP1,YY1,YY2,ZD1}; for a comprehensive list,
see the bibliography in~\cite{CK-tree}.

\item As explained in~\cite{CK-tree} (mostly but not only for trees), the
above general framework also includes the special case in which all $a_e
= 1$ -- for all graphs. In other words, $\dg = D_G^* - J$ is a rank-one
update of the multiplicative distance matrix $D_G^*$. In particular, our
data of interest $(\det(\cdot), \cof(\cdot))$ for $\dg$ and for $D_G^*$
can be recovered from each other via Lemma~\ref{Ldetcof}.

\item Now if we set $m_{ij} = q^{\alpha_{ij}}$ for $i,j \in V(G_e)$ (with
$\alpha_{ii} = 0\ \forall i \in [n]$), and reset $a_e := 1/(q-1)$ for all
$e \in E$, then we recover the \textit{$q$-distance matrix} $D_q(G)$ as a
special case of the additive-multiplicative distance matrix. In
particular, further setting $q \to 1$ recovers the classical distance
matrix of a possibly directed graph, with additive
edgeweights.\footnote{Observe, naively setting all $m_{ij} = 1$ instead,
yields a special case with trivial additive edgeweights: $\dg = 0$.} Note
that the classical and $q$-distance matrices are precisely the variants
studied in the literature for non-tree graphs: unicyclic, bicyclic,
polycyclic, cactus, and cycle-clique graphs; as well as hypertrees (i.e.\
graphs whose strong blocks are all complete).

\item We also remark that in the case of hypertrees, the special case in
which all off-diagonal entries $m_{ij}$ in a given block $G_e$ equal $q$,
and each $a_e$ equals $w_e/(q-1)$, was studied in~\cite{S-hypertrees},
and is not a $q$-distance matrix. In fact, to our knowledge the
closed-form expression for $\det(\dg)$ here is perhaps the one
\textit{known} formula\footnote{This work strictly extends all known
settings, so its results were of course not known or conjectured either.}
that our previous work~\cite{CK-tree} could not handle. This is addressed
by the present paper, whose results specialize to closed-form expressions
for $\det(\dg)$ as well as $\cof(\dg)$ for arbitrary hypertrees -- and in
the more general additive-multiplicative setting. See
Section~\ref{Shyper}.

\item Note in the previous example of~\cite{S-hypertrees} that the
inverse of the $q$-weighted distance matrix was not computed for
hypertrees. In this paper we are able to achieve this goal as well -- and
for an arbitrary \textit{additive-multiplicative} graph $G$ (in terms of
$\dge^{-1}$ for the strong blocks $G_e$ of $G$). This also subsumes our
formula in~\cite{CK-tree} for $\dt^{-1}$ for additive-multiplicative
trees.
\end{enumerate}
\end{remark}

\begin{remark}
If one varies the additive edgeweight inside a block, to try and work
with an even more general additive-multiplicative distance matrix, then
$\det(\dg)$ and $\cof(\dg)$ depend on these additive edgeweights
vis-a-vis the geometry of the blocks. This dependence was shown via an
example in~\cite[Section 1]{CK-tree} for trees. Similarly,
in~\cite{CK-tree} we defined an invariant $\kappa(\dg, v_0)$ which
depends on the location of the vertex $v_0$ when the additive edgeweight
is allowed to vary inside a block -- see the example in~\cite[Section
7]{CK-tree} of a triangle.
Below, we show that for all additive-multiplicative matrices,
$\kappa(\dg, v_0)$ is independent of $v_0$ -- and in fact, provides the
key ingredient to computing $\det(\dg), \cof(\dg)$ from the strong blocks
$G_e$ of $G$.
 In this sense also, is our additive-multiplicative model ``most
general'', at least to date.
\end{remark}

\begin{remark}[Zariski density]
A quick word on a \textit{novel technique} in this subfield, which we
introduced for trees in \cite{CK-tree}, and which is equally effective
and powerful for general graphs: Zariski density. This affords the
freedom to assume that several quantities that are not identically zero,
are ``always'' nonzero -- in fact invertible. This helps
significantly in computing determinants and inverses, and is
stated more precisely in Remark~\ref{Rzariski-short} and (in greater
detail) in previous work~\cite{CK-tree} (cf.\ \cite[Lemma 2.1]{Alon}).
\end{remark}

\noindent \textbf{Organization.}
We conclude by briefly listing the main results of this paper: one
theorem in each section below. In Section~\ref{S2}, we show our first
main result: a general set of Graham--Hoffman--Hosoya type identities,
which furthermore explain a novel graph invariant $\kappa(\dg)$,
introduced for trees in our recent work~\cite{CK-tree}. The identities
also have a host of applications, explained below.

In Section~\ref{S3}, we go even beyond computing the invariants
$\det(\cdot), \cof(\cdot), \kappa(\cdot)$ for arbitrary
additive-multiplicative $\dg$, to computing them for \textit{minors} of
$\dg$ -- and explain several cases when they too depend only on the
block-data. In Section~\ref{S4}, we compute $\dg^{-1}$ in closed-form,
and show that the traditionally used ``Laplacian matrix'' should be
corrected in the present, more general setting.

\section{{Theorem~\ref{Tmasterghh}:} The Master Graham--Hoffman--Hosoya
identities}\label{S2}

The Graham--Hoffman--Hosoya (GHH) identities~\cite{GHH} are the reason
why -- for any directed additively weighted tree $T$ -- the quantities
$\det(D_T), \cof(D_T)$ are independent of the tree structure and depend
only on the set of edgeweights (as mentioned in Section~\ref{S1}). As
shown in \textit{loc.~cit.}, these ``classical'' GHH identities hold for
all (finite simple directed strongly connected, as above) graphs $G$ with
\textit{additive} edgeweights, using their strong blocks.

The GHH identities in~\cite{GHH} were extended to a $q$-variant
in~\cite{S}. In our previous work~\cite{CK-tree} we obtained
generalizations of the original GHH identities that held for all
multiplicative distance matrices and all $q$-distance matrices (and in
particular specialized to the identities in \cite{GHH,S}). We also
provided at the end of~\cite{CK-tree}, a new invariant $\kappa(D_G, v_0)$
for additive-multiplicative trees and the corresponding GHH identities,
for general graphs $G$ rooted at a fixed cut-vertex $v_0$.

Now a natural question is if analogous identities hold in the present,
general additive-multiplicative setting. This turns out to be true, as is
now shown; moreover, a new invariant $\kappa(\dg)$ is crucially required.
Thus, our first result provides ``master'' Graham--Hoffman--Hosoya
identities, with the word ``master'' used for the following reasons:
\begin{enumerate}
\item The identities below explain the new invariant $\kappa(D_G,v_0)$
introduced in prior work~\cite{CK-tree}.

\item These (novel) identities apply to all additive-multiplicative
distance matrices of all graphs. Previous versions in~\cite{CK-tree, GHH,
S} were applicable only to restricted versions -- the multiplicative,
$q$, and classical distance matrices -- but not to the general version
$\dg$ (which was not even defined). For identities in our general
additive-multiplicative setting, $\kappa(\dg, v_0)$ is needed.

As a concrete special case: consider the formulas~\eqref{Emaster} from a
previous work~\cite{CK-tree}, for $\det(\dg)$ and $\cof(\dg)$, where $G$
is a tree. By inspection, these depend only on $\det(\dge)$ and
$\cof(\dge)$ for all
strong blocks -- i.e.~edges $e$ of $G$ -- but previous GHH variants
(including \cite{CK-tree}) could not explain why, for the
\textit{general} additive-multiplicative trees. Now we can.

\item The identities below specialize and apply to all of the previously
studied settings, to yield all previously known GHH identities. See
Corollary~\ref{Cmasterghh}.

\item In a later subsection, we will apply these identities to obtain
closed-form expressions for $\det(\dg), \cof(\dg)$ for arbitrary
additive-multiplicative hypertrees. This again extends prior formulas
in~\cite{CK-tree, S-hypertrees} to the additive-multiplicative setting,
which could not be handled by the results in previous papers (including
\textit{loc.\ cit.}).
\end{enumerate}

We begin with the main result in this section, which addresses the first
two of these points.

\begin{definition}[From~\cite{CK-tree}]\label{Dkappa}
Fix a vertex $v_0$ of a graph $G$. Given a subgraph $G'$ induced on a
subset of nodes $V'$ containing $v_0$, with additive-multiplicative
block-datum $\mathcal{G}'$, write
\[
D_{\mathcal{G}'} := (d(v,w))_{v,w \in V'} = \begin{pmatrix} D|_{V'
\setminus \{ v_0 \}} & {\bf u}_1 \\ {\bf w}_1^T & 0 \end{pmatrix},
\]
by relabelling the nodes, and define the invariant
\begin{equation}\label{Ekappa}
\kappa(D_{\mathcal{G}'}, v_0) := \det \left( D|_{V' \setminus \{ v_0 \}}
- {\bf u}_1 \, {\bf e}(V' \setminus \{ v_0 \})^T - {\bf m}(V' \setminus
\{ v_0 \}, v_0) {\bf w}_1^T \right).
\end{equation}
\end{definition}

For an alternate formula for $\kappa(\dg,v_0)$, see~\eqref{Enew-identity}
below (and see~\eqref{Ekappa2} for a generalization). Also, a notational
\textbf{remark}: in the sequel we may use multiple ${\bf e}$-vectors of
varying dimensions in our computations, the dimension of each being
deducible from context. At other times, we will use the notation ${\bf
e}(n)$ or ${\bf e}(V)$ depending on the context -- similarly for the
square matrix $J$ or $J_n$.

\begin{utheorem}[Master GHH identities]\label{Tmasterghh}
Fix a graph $G$ with node-set $[n]$, with strong blocks $G_e$ on $p_e$
nodes for $e \in E$. Let $\mathcal{G} = \{ \mathcal{G}_e = (a_e,
D^*_{G_e}) : e \in E \}$ be an additive-multiplicative block-datum, and
$\dg, D^*_G, \dge, D^*_{G_e}$ the corresponding matrices.
Also fix a (not necessarily cut) vertex $v_0$ of $G$. Then
$\kappa(\dg,v_0)$ is independent of $v_0$, and equals
\begin{equation}\label{Emghh1}
\kappa(\dg, v_0) = \det(D^*_G) \prod_{e \in E} a_e^{p_e - 1}.
\end{equation}
Writing this as $\kappa(\dg)$,
so that $\kappa(\dge) = a_e^{p_e - 1} \det(D^*_{G_e})$,
the following identities then hold:
\begin{align}
\kappa(\dg) = &\ \prod_{e \in E} \kappa(\dge),\label{Emghh2}\\
\frac{\det(\dg)}{\kappa(\dg)} = &\ \sum_{e \in E}
\frac{\det(\dge)}{\kappa(\dge)},\label{Emghh3}\\
\frac{\cof(\dg)}{\kappa(\dg)} - 1 = &\ \sum_{e \in E}
\left( \frac{\cof(\dge)}{\kappa(\dge)} - 1 \right),\label{Emghh4}
\end{align}
where the denominators in the second and third formulas are understood to
be placeholders, to be canceled by multiplying both sides by
$\kappa(\dg)$.
Finally, we have
\begin{equation}\label{Emghh5}
\cof(\dg) = \cof(D^*_G) \prod_{e \in E} a_e^{p_e - 1}, \qquad
\frac{\cof(\dg)}{\kappa(\dg)} = \frac{\cof(D^*_G)}{\det(D^*_G)}.
\end{equation}
\end{utheorem}


Before proving the theorem, we recall the special case of trees shown
in~\cite{CK-tree}, where each $G_e$ is a single edge $K_2$ (so $p_e =
2$), and $G = T$ has block-datum $\mathcal{T}$:
\begin{equation}\label{Emaster}
\det (\dt + x J)
= \prod_{e \in E} a_e (1 - m_e m'_e) \left[ \sum_{e
\in E} \frac{a_e (m_e - 1)(m'_e - 1)}{m_e m'_e - 1} + x \left( 1 -
\sum_{e \in E} \frac{(m_e - 1)(m'_e - 1)}{m_e m'_e - 1} \right)
\right].
\end{equation}

\noindent Observe two points about these, unexplained in~\cite{CK-tree},
which can now be understood via Theorem~\ref{Tmasterghh}:
\begin{itemize}
\item The common multiplicative factor $\prod_e (a_e (1 - m_e m'_e))$ is
precisely $\kappa(\dt)$ (at any vertex $v_0$), since the set $E$ of
strong blocks is precisely the edge-set of the tree, and each $D^*_{G_e}
= \begin{pmatrix} 1 & m_e \\ m'_e & 1 \end{pmatrix}$ (up to transposing).
Now the formulas for $\displaystyle \frac{\det(\dt)}{\kappa(\dt)}$ and $1
- \displaystyle \frac{\cof(\dt)}{\kappa(\dt)}$ are remarkably similar:
they involve summing the same terms, with an additional factor of $a_e$
in the former expression.

\item $\cof(\dt)$ does not depend on the complete datum $\mathcal{T} = \{
(a_e, D^*_{G_e}) : e \in E \} = \{ (a_e, m_e, m'_e) : e \in E \}$, but
rather on the decoupled data $\{ a_e \}$ and $\{ D^*_{G_e} \} = \{ (m_e,
m'_e) \}$.
\end{itemize}

We now explain why both phenomena are special cases of our master GHH
identities. In the first case, since $\dge = a_e (D^*_{G_e} - J)$, the
summands in the right-hand sides of~\eqref{Emghh3} and~\eqref{Emghh4} are
computable by Lemma~\ref{Ldetcof} to be, respectively:
\begin{equation}\label{Etemp}
\frac{\det(\dge)}{\kappa(\dge)} = a_e \left( 1 -
\frac{\cof(D^*_{G_e})}{\det(D^*_{G_e})} \right), \qquad
\frac{\cof(\dge)}{\kappa(\dge)} - 1 =
\frac{\cof(D^*_{G_e})}{\det(D^*_{G_e})} - 1.
\end{equation}
This explains the first phenomenon above, but now for all $G$ -- as well
as the second phenomenon, since the latter quantity in~\eqref{Etemp} does
not depend on $a_e$. In particular,~\eqref{Emghh3} can be restated as:
\[
\frac{\det(\dg)}{\kappa(\dg)} = \sum_{e \in E} a_e
\left( 1 - \frac{\cof(\dge)}{\kappa(\dge)} \right).
\]

A pleasing consequence of our master GHH formulas -- given the present
discussion -- is that even for arbitrary graphs, $\det(\dg), \cof(\dg),
\kappa(\dg)$ do not depend on the complete datum $\mathcal{G} = \{ (a_e,
D^*_{G_e}) : e \in E \}$, but once again only on triples and pairs:

\begin{cor}
For any additive-multiplicative distance matrix $\dg$ of an arbitrary
graph $G$,
\begin{enumerate}
\item $\det(\dg)$ depends on the datum $\mathcal{G}$ only through the
triples $\{ (a_e, \det(D^*_{G_e}), \cof(D^*_{G_e})) : e \in E \}$,
\item $\cof(\dg)$ through the decoupled data $\{ a_e : e \in E \} \sqcup
\{ (\det(D^*_{G_e}), \cof(D^*_{G_e})) : e \in E \}$, and \item
$\kappa(\dg)$ depends on $\mathcal{G}$ through even less: $\{ a_e : e \in
E \} \sqcup \{ \det(D^*_{G_e}) : e \in E \}$.
\end{enumerate}
\end{cor}

\begin{proof}[Proof of Theorem~\ref{Tmasterghh}]
(We will deduce all previous GHH identities in
Corollary~\ref{Cmasterghh}.)
We first assume~\eqref{Emghh1} and show that it implies the rest.
Clearly~\eqref{Emghh1} implies~\eqref{Emghh2}. We now show~\eqref{Emghh3}
and~\eqref{Emghh4} by induction on the number $|E|$ of strong blocks in
$G$. For $|E|=1$ the identities are tautologies; otherwise the graph $G$
has a ``pendant block'' (or ``end-block'') $G_e$, separated by a
cut-vertex $v_0$ from the rest of the subgraph; call it $G'$. But then by
\cite[Theorem D]{CK-tree}, we have
\begin{align*}
\frac{\det(\dg)}{\kappa(\dg, v_0)} = &\ 
\frac{\det(\dge)}{\kappa(\dge, v_0)} +
\frac{\det(D_{\mathcal{G}'})}{\kappa(D_{\mathcal{G}'}, v_0)},\\
\frac{\cof(\dg)}{\kappa(\dg,v_0)} - 1 = &\ 
\left( \frac{\cof(\dge)}{\kappa(\dge, v_0)} - 1 \right) +
\left( \frac{\cof(D_{\mathcal{G}'})}{\kappa(D_{\mathcal{G}'}, v_0)} - 1
\right).
\end{align*}
The last term on the right in both equations has denominator
$\kappa(D_{\mathcal{G}'})$ by the (assumed) identity~\eqref{Emghh1}, so
we are now done by the induction hypothesis, since $G'$ has fewer blocks
than $G$.

Finally, the identities in~\eqref{Emghh5} are shown as follows:
by~\eqref{Emghh4} and~\eqref{Etemp}, we obtain:
\[
\frac{\cof(\dg)}{\kappa(\dg)} - 1 = \sum_{e \in E} \left( 
\frac{\cof(D^*_{G_e})}{\det(D^*_{G_e})} - 1 \right).
\]
In particular, this equation also holds for the special case when $a_e =
1\ \forall e$. Here the right-hand side is unchanged, while as stated in
Remark~\ref{Rspecial}(2), now $\dg = D^*_G - J$, so the left-hand side is
\[
\frac{\cof(D^*_G - J)}{\det(D^*_G)} - 1 = 
\frac{\cof(D^*_G)}{\det(D^*_G)} - 1
\]
by Lemma~\ref{Ldetcof}. This shows the second of the identities
in~\eqref{Emghh5}, and hence the first.

\begin{remark}\label{Rzariski-short}
In this paper, the freedom to work with denominators that can vanish at
special values of the parameters (in an arbitrary commutative ring $R$)
is permissible via Zariski density arguments -- a technique introduced
in such settings by our prior work~\cite{CK-tree}. Concretely: note by
specializing that $\kappa(\dge, v_0), \kappa(D_{\mathcal{G}'}, v_0),
\kappa(\dg, v_0)$ are nonzero polynomials in the $a_e$ and the
off-diagonal entries of $D^*_{G_e}$ (over all $e \in E$). Thus we first
work over the field of rational functions in all of these variables --
with coefficients in $\mathbb{Q}$ -- and show that all assertions hold on
the nonzero-locus of a finite set of nonzero polynomials over an infinite
field. By Zariski density, the assertions hold on the entire affine
space. Now the equality of the two sides (in all claimed
identities/assertions) holds in the ring of polynomials in the same
variables (since there are no denominators) -- and with coefficients in
$\Z$. Finally, specialize the parameters to take values in $R$. We refer
the reader to~\cite{CK-tree} for specific examples of this procedure, and
omit such demonstrations in this paper for brevity.
\end{remark}

We now come to the meat of the proof, which is in showing~\eqref{Emghh1},
and we again do so by induction on the number $|E|$ of strong blocks in
$G$. For $E$ a singleton, write $\dg$ in the form
\[
\dg = a_1 (D^*_G - J) = a_1 \begin{pmatrix} D^*_1 - J & {\bf u}_1 - {\bf
e} \\ {\bf w}_1^T - {\bf e}^T & 0 \end{pmatrix},
\]
under a suitable ordering of the nodes (i.e.~with $v_0$ last), and where
${\bf e} = {\bf e}(V(G) \setminus \{ v_0 \})$. Then,
\[
\kappa(\dg, v_0) = a_1^{p_1 - 1} \det( D^*_1 - J -  ({\bf u}_1 - {\bf e})
{\bf e}^T - {\bf u}_1 ({\bf w}_1^T - {\bf e}^T)) = a_1^{p_1 - 1} \det
\begin{pmatrix} D^*_1 & {\bf u}_1\\ {\bf w}_1^T  & 1 \end{pmatrix},
\]
as desired. For the induction step, the first case is when $v_0$ a
cut-vertex of $G$. Let $G = G_1 \sqcup_{v_0} G_2$, with $V(G_1) = \{ 1,
\dots, v_0 \}$ and $V(G_2) = \{ v_0, \dots, n \}$. Set $V'_j := V(G_j)
\setminus \{ v_0 \}$ and write
\[
\dg := \begin{pmatrix}
a_1 (D_1^* - J) & a_1 ({\bf u}_1  - {\bf e}) & a_1 ({\bf u}_1 - {\bf
e}){\bf e}^T + a_2 {\bf u}_1 ({\bf w}_2^T - {\bf e}^T) \\
a_1 ({\bf w}_1^T - {\bf e}^T) & 0 & a_2 ({\bf w}_2^T - {\bf e}^T) \\
a_2 ({\bf u}_2 - {\bf e}) {\bf e}^T + a_1 {\bf u}_2 ({\bf w}_1^T - {\bf
e}^T) & a_2 ({\bf u}_2 - {\bf e}) & a_2 (D_2^* - J)
\end{pmatrix}.
\]
Then as asserted in the proof of \cite[Theorem D]{CK-tree}, a
straightforward computation shows that
\[
\kappa(\dg, v_0) = \kappa(D_{\mathcal{G}_1}, v_0)
\kappa(D_{\mathcal{G}_2}, v_0).
\]
In particular, if $v_0$ is a cut-vertex of $G$ then~\eqref{Emghh1}
follows by the induction hypothesis.

It remains to consider the case when $v_0$ is not a cut-vertex. Let
$G_{e_0} = G_0$ be the strong block of $G$ containing $v_0$, and let
$v_1, \dots, v_k$ be the cut-vertices of $G$ lying in $G_0$. Let $G_1,
\dots, G_k$ denote the maximal (pairwise disjoint) induced subgraphs of
$G$ such that $G_j$ intersects $G_0$ precisely at $v_j$:
\[
G = G_0 \sqcup_{v_1} G_1 \sqcup_{v_2} G_2 \cdots \sqcup_{v_k} G_k.
\]
For this graph, order the vertices in $k+1$ groups, each with a vertex
followed by a set of vertices:
\begin{equation}\label{Enodes}
v_0, \ V'_0 := V(G_0) \setminus \{ v_j \}; \quad v_1, \  V'_1 := V(G_1)
\setminus \{ v_1 \}; \quad \dots; \quad v_k, \ V'_k := V(G_k) \setminus
\{ v_k \}.
\end{equation}
Corresponding to these, we write $\dg$ as a $(2k+2) \times (2k+2)$ block
matrix, consisting of $(k+1)^2$-many $2 \times 2$ block matrices
$B_{ij} : 0 \leq i,j \leq k$, where
\begin{equation}
B_{ij} := \begin{cases}
a_i \begin{pmatrix}
0 & {\bf w}_i^T - {\bf e}^T \\ {\bf u}_i - {\bf e} & D_i - J
\end{pmatrix},
& \text{if } i = j \in \{ 0, \dots, k \},\\
\begin{pmatrix}
a_0 (m_{v_0 v_j} - 1) & a_0 (m_{v_0 v_j} - 1) {\bf e}^T + a_j m_{v_0
v_j} ({\bf w}_j^T - {\bf e}^T) \\
a_0 ({\bf m}_{V'_0, v_j} - {\bf e}) & 
a_0 ({\bf m}_{V'_0, v_j} - {\bf e}) {\bf e}^T +  a_j {\bf m}_{V'_0, v_j}
({\bf w}_j^T - {\bf e}^T)
\end{pmatrix},
& \text{if } 0 = i < j \leq k,\\
\begin{pmatrix}
a_0 (m_{v_i v_0} - 1) & a_0 ({\bf m}_{v_i, V'_0}^T - {\bf e}^T) \\
a_i ({\bf u}_i - {\bf e}) + a_0 (m_{v_i v_0} - 1) {\bf u}_i &
a_i ({\bf u}_i - {\bf e}) {\bf e}^T + a_0 {\bf u}_i ({\bf m}_{v_i,
V'_0}^T - {\bf e}^T) \\
\end{pmatrix},
& \text{if } 0 = j < i \leq k,\\
\begin{pmatrix}
a_0 (m_{v_i v_j} - 1) & \vline & a_0 (m_{v_i v_j} - 1) {\bf
e}^T + a_j m_{v_i, v_j} ({\bf w}_j^T - {\bf e}^T)\\
\hline
a_i ({\bf u}_i - {\bf e}) + & \vline & a_i ({\bf u}_i - {\bf e}) {\bf
e}^T + a_0 (m_{v_i v_j} - 1) {\bf u}_i {\bf e}^T + \\
a_0 (m_{v_i v_j} - 1) {\bf u}_i & \vline &
a_j m_{v_i v_j} {\bf u}_i ({\bf w}_j^T - {\bf e}^T)
\end{pmatrix},
& \text{if } 1 \leq i,j \leq k, \ i \neq j.
\end{cases}
\end{equation}

Recall here that $m_{v_i v_j}$ for $0 \leq i,j \leq k$ is the
corresponding entry in the multiplicative distance matrix $D^*_{G_0}$.
Now write $\dg = ( B_{ij} )_{i,j=0}^k$ in the form $\begin{pmatrix} 0 &
{\bf w}^T \\ {\bf u} & D \end{pmatrix}$; as a result,
\[
\kappa(\dg, v_0) = \det(D - {\bf u} {\bf e}^T - {\bf m}(V \setminus \{
v_0 \}, v_0) {\bf w}^T).
\]
Next, ${\bf u}^T, {\bf m}(V \setminus \{ v_0 \})^T, {\bf w}^T$ have
$V'_0; v_i, V'_i$ components (for $i \in [k]$) given respectively by:
\begin{align}
{\bf u}^T = &\ (a_0 ({\bf u}_0^T - {\bf e}^T); \quad
(a_0 (m_{v_i v_0} - 1), \ \
a_i ({\bf u}_i^T - {\bf e}^T) + a_i (m_{v_i v_0} - 1) {\bf
u}_i^T)_{i=1}^k),\notag\\
{\bf m}(V \setminus \{ v_0 \}, v_0)^T = &\ ({\bf u}_0^T; \quad
(m_{v_i v_0}, \ \ m_{v_i v_0} {\bf u}_i^T)_{i=1}^k),\\
{\bf w}^T = &\ (a_0 ({\bf w}_0^T - {\bf e}^T); \quad
(a_0 (m_{v_0 v_j} - 1), \ \ 
a_0 (m_{v_0 v_j} - 1) {\bf e}^T + a_j m_{v_0 v_j} ({\bf w}_j^T - {\bf
e}^T))_{j=1}^k).\notag
\end{align}

From this, we obtain the matrix 
$A_\mathcal{G} := {\bf u} {\bf e}^T + {\bf m}(V \setminus \{ v_0 \}, v_0)
{\bf w}^T)$ in block form as follows:
\begin{equation}
A_\mathcal{G} = \begin{pmatrix}
a_0 ({\bf u}_0 {\bf w}_0^T - J) & \vline &
a_0 (m_{v_0 v_j} {\bf u}_0 - {\bf e}) & \vline &
a_0 (m_{v_0 v_j} {\bf u}_0 - {\bf e}) {\bf e}^T +\\
& \vline & & \vline & a_j m_{v_0 v_j} {\bf u}_0 ({\bf w}_j^T - {\bf
e}^T)\\
\hline
a_0 (m_{v_i v_0} {\bf w}_0^T - {\bf e}^T) & \vline &
a_0 (m_{v_i v_0} m_{v_0 v_j} - 1) & \vline &
a_0 (m_{v_i v_0} m_{v_0 v_j} - 1) {\bf e}^T +\\
& \vline & & \vline & a_j m_{v_i v_0} m_{v_0 v_j} ({\bf w}_j^T - {\bf
e}^T)\\
\hline
a_i ({\bf u}_i - {\bf e}) {\bf e}^T + & \vline &
a_i ({\bf u}_i - {\bf e}) + & \vline &
a_i ({\bf u}_i - {\bf e}) {\bf e}^T +\\
a_0 {\bf u}_i (m_{v_i v_0} {\bf w}_0^T - {\bf e}^T) & \vline &
a_0 (m_{v_i v_0} m_{v_0 v_j} - 1) {\bf u}_i & \vline &
a_0 (m_{v_i v_0} m_{v_0 v_j} - 1) {\bf u}_i {\bf e}^T +\\
& \vline & & \vline & a_j m_{v_i v_0} m_{v_0 v_j} {\bf u}_i ({\bf w}_j^T
- {\bf e}^T)
\end{pmatrix}.
\end{equation}
Here, the rows are labelled by $V'_0; \ v_i, \ V'_i$ for $i \in [k]$, and
the columns by $V'_0; \ v_j, \ V'_j$ for $j \in [k]$. Note that the
square matrix $A_\mathcal{G}$ has dimension $|V(G)|-1$. Now the above
computations yield:
\[
\kappa(\dg, v_0) = \det(\dg - A_\mathcal{G}) = \det
((B'_{ij})_{i,j=0}^k),
\]
where $B'_{ij}$ has the same size as $B_{ij}$ for $i,j>0$, and $B'_{i0},
B'_{0j}$ have one less row and one less column respectively for all $0
\leq i,j \leq k$. Moreover, the $B'_{ij}$ are given by:
\begin{equation}
B'_{ij} := \begin{cases}
a_0 (D_0 - {\bf u}_0 {\bf w}_0^T),
& \text{if } i = j = 0,\\
({\bf m}_{V'_0,v_j} - m_{v_0 v_j} {\bf u}_0)
\begin{pmatrix} a_0 & a_0 {\bf e}^T + a_j ({\bf w}_j^T - {\bf
e}^T) \end{pmatrix},
& \text{if } 0 = i < j \leq k,\\
a_0 \begin{pmatrix} 1 \\ {\bf u}_i \end{pmatrix}
({\bf m}_{v_i,V'_0} - m_{v_i v_0} {\bf w}_0)^T,
& \text{if } 0 = j < i \leq k,\\
\gamma_{ij} \begin{pmatrix} 1 \\ {\bf u}_i \end{pmatrix}
\begin{pmatrix} a_0 & a_0 {\bf e}^T + a_j ({\bf w}_j^T - {\bf e}^T)
\end{pmatrix} + \delta_{ij} \begin{pmatrix} 0 & 0 \\ 0 & a_i (D_i - {\bf
u}_i {\bf w}_i^T) \end{pmatrix}, \quad
& \text{if } 1 \leq i,j \leq k,
\end{cases}
\end{equation}
where $\delta_{ij}$ denotes the Kronecker delta, and $\gamma_{ij} :=
m_{v_i v_j} - m_{v_i v_0} m_{v_0 v_j}$ for all $i,j \in [k]$, with the
understanding that $m_{v_i v_i} = 1$ in our notation/convention.

Now to compute $\det((B'_{ij})_{i,j=0}^k)$, we perform block row
operations, in which for each $i \in [k]$, from the $V'_i$-block row we
subtract ${\bf u}_i$ times the $v_i$-block row. This kills all entries
except the $(V'_i, V'_i)$-block entry, which is $a_i (D_i - {\bf u}_i
{\bf w}_i^T)$. By e.g.~``block upper triangularity'', it follows that
\begin{equation}\label{Etoprove}
\kappa(\dg, v_0) = \det(D - A_\mathcal{G}) = \prod_{i=1}^k \det(a_i
(D_i - {\bf u}_i {\bf w}_i^T)) \cdot \det A_0,
\end{equation}
for a particular square matrix $A_0$ (given below) of dimension
$|V(G_0)|-1$. Also note that since the graph~$G_i$ (with node set $V'_i
\sqcup \{ v_i \}$) has a smaller number of blocks, the induction
hypothesis yields
\[
\det(a_i (D_i - {\bf u}_i {\bf w}_i^T)) = \kappa(D_{\mathcal{G}_i}, v_i)
= \kappa(D_{\mathcal{G}_i}), \qquad \forall i \in [k].
\]

It thus remains to show that $\det(A_0) = a_0^{|V(G_0)|-1}
\det(D^*_{G_0}) = \kappa(D_{\mathcal{G}_0}, v_0) =
\kappa(D_{\mathcal{G}_0})$, where
\[
A_0 := a_0 \begin{pmatrix}
D_0 - {\bf u}_0 {\bf w}_0^T & {\bf m}_{V'_0,v_1} - m_{v_0 v_1} {\bf u}_0
& {\bf m}_{V'_0,v_2} - m_{v_0 v_2} {\bf u}_0 & \cdots & {\bf
m}_{V'_0,v_k} - m_{v_0 v_k} {\bf u}_0\\
{\bf m}_{v_1, V'}^T - m_{v_1 v_0} {\bf w}_0^T & 
\gamma_{11} & \gamma_{12} & \cdots & \gamma_{1k}\\
{\bf m}_{v_2, V'}^T - m_{v_2 v_0} {\bf w}_0^T & 
\gamma_{21} & \gamma_{22} & \cdots & \gamma_{2k}\\
\vdots & \vdots & \vdots & \ddots & \vdots\\
{\bf m}_{v_k, V'}^T - m_{v_k v_0} {\bf w}_0^T & 
\gamma_{k1} & \gamma_{k2} & \cdots & \gamma_{kk}
\end{pmatrix}.
\]
But it is straightforward to verify that if one writes $D^*_{G_0}$ with
the same labelling of the nodes $v_0, V'_0$, and subtracts a suitable
multiple of the first row from every other row (to cancel their leading
entries) -- or does the analogous column operations -- then the matrix
one obtains has leading column ${\bf e}_1$ (or row ${\bf e}_1^T$), and
the principal submatrix obtained by now removing the first row and column
is precisely $a_0^{-1} A_0$. Hence $\det(A_0) = a_0^{|V(G_0)|-1}
\det(D^*_{G_0})$, and by~\eqref{Etoprove} and the induction hypothesis,
the proof is complete.
\end{proof}

\subsection{Application 1: Previous GHH identities}

Theorem~\ref{Tmasterghh} has several applications; some are now listed.
First, the result explains all previously known GHH identities, including
in our recent work~\cite{CK-tree}. We begin with a straightforward
reformulation of the GHH identities in Theorem~\ref{Tmasterghh}:

\begin{prop}\label{Pmasterghh}
Notation as in Theorem~\ref{Tmasterghh}. Then:
\begin{align}
\det(\dg) = &\ \sum_{e \in E} \det(\dge) \prod_{f \neq e}
a_f^{p_f - 1} \det(D^*_{G_f}),\\
\cof(\dg) = &\ \det(D^*_G) \prod_{e \in E} a_e^{p_e - 1} + \sum_{e \in E}
(\cof(\dge) - a_e^{p_e - 1} \det(D^*_{G_e})) \prod_{f \neq e}
a_f^{p_f - 1} \det(D^*_{G_f}).
\end{align}
\end{prop}

As a consequence, we obtain the following application of our master GHH
identities:

\begin{cor}\label{Cmasterghh}
Proposition~\ref{Pmasterghh} implies the GHH identities for
multiplicative, $q$, and classical distance matrices, shown
in~\cite{CK-tree}, \cite{CK-tree}, and~\cite{GHH} respectively (following
previously shown special cases):
\begin{gather}
\begin{aligned}\label{Eghh12}
\hspace*{2mm}\det(D^*_G) = &\ \prod_{e \in E} \det(D^*_{G_e}),\\
\cof(D^*_G) = &\ \prod_{e \in E} \det(D^*_{G_e}) + \sum_{e \in E}
(\cof(D^*_{G_e}) - \det(D^*_{G_e})) \prod_{f \neq e} \det(D^*_{G_f}),
\end{aligned}\\
\begin{aligned}\label{Eghh34}
\det(D_q(G)) = &\ \sum_{e \in E} \det(D_q(G_e)) \prod_{f \neq e} d^*_f,\\
\cof(D_q(G)) = &\ \prod_{e \in E} d^*_e - (q-1) \sum_{e \in E}
\det(D_q(G_j)) \prod_{f \neq e} d^*_f,
\end{aligned}\\
\begin{aligned}\label{Eghh56}
\det(D_G) = &\ \sum_{e \in E} \det(D_{G_e}) \prod_{f \neq e}
\cof(D_{G_f}),\\
\cof(D_G) = &\ \prod_{e \in E} \cof(D_{G_e}).
\end{aligned}
\end{gather}
Here $D_q(G)$ is defined in Remark~\ref{Rspecial}(3), $D_G := D_1(G)$,
and $d^*_e := (q-1) \det(D_q(G_e)) + \cof(D_q(G_e))$.
\end{cor}

As mentioned above, there are also GHH identities involving the invariant
$\kappa(\dg, v_0)$; these are subsumed by our main result in this section
-- see Theorem~\ref{Tmasterghh}.

\begin{proof}
Setting all $a_e = 1$, we observe that $D^*_G = \dg - J$ (see
e.g.~Remark~\ref{Rspecial}(2)).
Now the first two identities~\eqref{Eghh12} are easy consequences of
Proposition~\ref{Pmasterghh}, via Lemma~\ref{Ldetcof}. That these
identities~\eqref{Eghh12} for $D^*_G$ imply the identities~\eqref{Eghh34}
for $D_q(G)$ was shown in \cite[Section 2]{CK-tree}; and
specializing~\eqref{Eghh34} to $q \to 1$ yields the classical GHH
identities~\eqref{Eghh56} for $D_G = D_1(G)$.
\end{proof}

\subsection{Application 2: Hypertrees}\label{Shyper}

Our next application of Theorem~\ref{Tmasterghh} is to
\textit{hypertrees} -- graphs whose strong blocks are cliques. Here we
provide a novel additive-multiplicative model, which generalizes and
subsumes the known variants.
In our model, every block of the hypertree $G$ is a bidirected clique
$K_p$, with vertex set $V = [p] = \{ 1, \dots, p \}$. We begin with the
multiplicative component, in which every edge $i \to j$ is equipped with
both a ``head'' and ``tail'' contribution:
\begin{equation}
(D^*_{K_p})_{i \to j} := \begin{cases}
m_i m'_j, \qquad & \text{if } i \neq j,\\
1 & \text{otherwise}.
\end{cases}
\end{equation}

For example, one can take all $m_i = q, m'_i = 1$, leading to the
$q$-exponential distance matrix $D^*_{K_p}$ (hence to the $q$-distance
matrix $D_q(K_p)$ if we choose $a = 1/(q-1)$). We now have:

\begin{lemma}\label{Lhypertree-mult}
If ${\bf d}$ denotes the vector with $v$th component
$\frac{m_v m'_v}{1 - m_v m'_v}$, then 
\begin{align}
\begin{aligned}
\det D^*_{K_p} = &\ \prod_{v \in V} (1 - m_v m'_v) \cdot ( 1 + {\bf
e}(p)^T {\bf d} ),\\
\cof D^*_{K_p} = &\ \prod_{v \in V} (1 - m_v m'_v) \cdot \left( p +
{\bf e}(p)^T {\bf d} + \sum_{v<w} \frac{(m_v - m_w) (m'_v - m'_w)}{(1 - m_v
m'_v) (1 - m_w m'_w)} \right).
\end{aligned}
\end{align}
\end{lemma}

The denominators in both right-hand sides (including in ${\bf d}$) are
understood to be placeholders, which cancel with factors of the products
on the right.

\begin{proof}
We sketch a proof. Let ${\bf m'} := (m'_v)_{v \in V}$. Then
$D^*_{K_p} = \diag (1 - m_v m'_v)_{v \in V} + {\bf m} ({\bf m'})^T$.
This easily yields $\det D^*_{K_p}$.
Next, the Sherman--Morrison formula for a rank-one update implies:
\[
(D^*_{K_p})^{-1} = \diag({\bf m'})^{-1} \cdot C \cdot \diag({\bf
m})^{-1}, \qquad \text{where }
C := \diag({\bf d}) - \frac{1}{1 + {\bf e}(p)^T {\bf d}} {\bf d} {\bf
d}^T,
\]
and from this one obtains $\cof(D^*_{K_p}) = \det(D^*_{K_p}) \cdot {\bf
e}(p)^T (D^*_{K_p})^{-1} {\bf e}(p)$, if $D^*_{K_p}$ is invertible. The
general case now follows by Zariski density, see
e.g.~Remark~\ref{Rzariski-short} or proofs of special cases
in~\cite{CK-tree}.
\end{proof}

From these computations, one derives the desired invariants for arbitrary
hypertrees:

\begin{prop}\label{Phypertree}
Suppose $G$ is a hypertree with strong blocks $G_e = K_{p_e}$ for $e \in
E$. Suppose the nodes of each clique are labelled by $[p_e]$, with
block-data given by $a_e, m_{e,v}, m'_{e,v}$ as above. Also define the
vector
$\displaystyle {\bf d}_e := \left(\frac{m_{e,v} m'_{e,v}}{1 - m_{e,v}
m'_{e,v}} \right)_{v \in [p_e]}$
for all blocks $e \in E$. Then,
\begin{align*}
\kappa(\dg) = &\ \prod_{e \in E} \left( a_e^{p_e - 1} (1 + {\bf e}(p_e)^T
{\bf d}_e) \prod_{v \in [p_e]} (1 - m_{e,v} m'_{e,v}) \right),\\
\det(\dg) = &\ \kappa(\dg) \left[ \sum_{e \in E} \frac{-a_e}{1 + {\bf
e}(p_e)^T {\bf d}_e} \left( p_e - 1 + \sum_{v<w \in [p_e]} \frac{(m_{e,v}
- m_{e,w}) (m'_{e,v} - m'_{e,w})}{(1 - m_{e,v} m'_{e,v}) (1 - m_{e,w}
m'_{e,w})} \right) \right],\\
\cof(\dg) = &\ \kappa(\dg) \left[ 1 + \sum_{e \in E} \frac{1}{1 + {\bf
e}(p_e)^T {\bf d}_e} \left( p_e - 1 + \sum_{v<w \in [p_e]} \frac{(m_{e,v}
- m_{e,w}) (m'_{e,v} - m'_{e,w})}{(1 - m_{e,v} m'_{e,v}) (1 - m_{e,w}
m'_{e,w})} \right) \right].
\end{align*}
\end{prop}

\begin{proof}
With the present notation, the additive-multiplicative matrix of the
block~$G_e$ is precisely $\dge := a_e (D^*_{G_e} - J)$ as above, for
each $e \in E$. For this matrix, Lemmas~\ref{Lhypertree-mult}
and~\ref{Ldetcof} imply:
\begin{align*}
\det(\dge + x J)
= &\ a_e^{p_e - 1} \prod_{v \in [p_e]} (1 - m_{e,v} m'_{e,v})\\
&\ \times \left[ x \cdot {\bf e}(p_e)^T {\bf d}_e + a_e + (x-a_e) \left(
p_e + \sum_{v<w} \frac{(m_{e,v} - m_{e,w}) (m'_{e,v} -
m'_{e,w})}{(1 - m_{e,v} m'_{e,v}) (1 - m_{e,w} m'_{e,w})} \right)
\right].
\end{align*}
Now apply Lemma~\ref{Lhypertree-mult} and Theorem~\ref{Tmasterghh} to
complete the calculations.
\end{proof}

As a sample corollary, specialize $m_{e,v} = m'_{e,v} = \sqrt{q}, a_e =
w_e/(q-1)$ for all $e \in E$ and $v \in [p_e]$. This yields the invariants
$\det(\dg), \cof(\dg), \kappa(\dg)$ for $\dg = D^w_q(G)$, the weighted
$q$-distance matrix of a hypertree -- both for general $q$ and for the
``classical'' case $q=1$:

\begin{cor}\label{Chypertree}
Let $G$ be a hypertree with strong blocks $K_{p_e}$, additive weights
$w_e / (q-1)$ for $e \in E$, and all multiplicative weights $m_{e,v} =
m'_{e,v} = \sqrt{q}$. Then,
\begin{align*}
\kappa(\dg) = &\ \prod_{e \in E} (-w_e)^{p_e - 1} (1 + (p_e - 1)q),\\
\det(\dg + xJ) = &\ \kappa(\dg) \left( x + \sum_{e \in E}
\frac{(p_e - 1)(w_e + x (1-q))}{1 + (p_e - 1)q} \right).
\end{align*}
\end{cor}

Recall~\cite{S-hypertrees} that a hypertree is \textit{$d$-regular} if
every strong block is a clique on $d$ nodes. The $d=2$ case (i.e., trees)
was worked out in the above generality in our prior work~\cite{CK-tree},
while Sivasubramanian~\cite{S-hypertrees} computed $\det(\dg)$, in the
special case $m_v = m'_v = \sqrt{q}$ for all blocks. Both results -- and
hence all prior variants for (hyper)trees -- are subsumed by
Corollary~\ref{Chypertree}, so by Proposition~\ref{Phypertree}.

\subsection{Application 3: Adding pendant hypertrees}

Our final application in this section extends a prior result of
Bapat--Kirkland--Neumann (for trees) and its subsequent generalizations
(e.g.~in~\cite{CK-tree} for additive-multiplicative trees) to
arbitrary hypertrees:

\begin{prop}\label{Paddcliques}
Let $k, p'_1, \dots, p'_k \geq 1$ be integers. For each $j \in [k]$ let
$G_j$ be a weighted bi-directed graph with node set $[p'_j]$, no
cut-vertices, and with additive-multiplicative distance matrix
$D_{\mathcal{G}_j} = a_j (D^*_{G_j} - J_{p'_j})$ that satisfies:
\begin{equation}\label{Eevector}
D_{\mathcal{G}_j} {\bf e}(p'_j) = d_j {\bf e}(p'_j), \qquad \forall j \in
[k],
\end{equation}
where $d_j \in R$ are scalars.
Now let $G'$ be any graph with strong blocks $G_1, \dots, G_k$, and let
$G$ be obtained from $G'$ by further attaching finitely many cliques
inductively -- i.e.~to a vertex of the graph constructed at each step, so
that these cliques are also strong blocks. Let these cliques have
block-data as in Proposition~\ref{Phypertree} (so they do not necessarily
satisfy~\eqref{Eevector}).

Then $\det(\dg), \cof(\dg)$ (and $\kappa(\dg)$) depend not on the
structure and locations of the attached cliques $G_e$, $e \in E$ or the
blocks $G_j$, $j \in [k]$ but only on their block-data, and as follows:
\begin{align}\label{Eaddcliques}
&\ \det(\dg + x J)\\
= &\ \prod_{j=1}^k \left( \det(D_{\mathcal{G}_j}) (a_j^{-1} +
(p'_j/d_j)) \right) \prod_{e \in E} \left( a_e^{p_e - 1} (1 + {\bf
e}(p_e)^T {\bf d}_e) \prod_{v \in [p_e]} (1 - m_{e,v} m'_{e,v}) \right)
\notag\\
&\ \times \left[ x + \sum_{j=1}^k \frac{a_j - x}{1 + \frac{a_j
p'_j}{d_j}} + \sum_{e \in E} \frac{x-a_e}{1 + {\bf e}(p_e)^T {\bf d}_e}
\left( p_e - 1 + \sum_{v<w \in [p_e]} \frac{(m_{e,v} - m_{e,w}) (m'_{e,v}
- m'_{e,w})}{(1 - m_{e,v} m'_{e,v}) (1 - m_{e,w} m'_{e,w})} \right)
\right]. \notag
\end{align}
\end{prop}

We make three concluding remarks: first, the proof of
Proposition~\ref{Paddcliques} involves straightforward computations using
Theorem~\ref{Tmasterghh} and Proposition~\ref{Phypertree}, hence is
omitted.
Second, the $a_j^{-1}$ in the first product on the right is a placeholder
(when working over an arbitrary commutative unital ring).
Indeed, the entire factor in the first product is precisely
$\kappa(D_{\mathcal{G}_j}) = a_j^{p_j - 1} \det(D^*_{G_j})$.

Third, Proposition~\ref{Paddcliques} subsumes all of the known results in
the literature along these lines, including for trees (see the results
and references in~\cite{CK-tree}) as well as
Proposition~\ref{Phypertree} -- but also covering the special case of
cycle-clique graphs. These are graphs whose strong blocks are cliques or
cycles, the latter of which fit the description of the $G_j$ -- in that
they have ${\bf e}$ as an eigenvector for $D_{\mathcal{G}_j}$.
Cycle-clique graphs include unicyclic, bicyclic, polycyclic (including
cactus) graphs as well as (hyper)trees -- see~\cite[Section 2]{CK-tree}
and the references therein for the numerous special cases studied in the
literature. A final comment is that the literature studied only the
special case of $D_q(G)$ (the $q$-distance matrix); but our result holds
more generally for \textit{all} additive-multiplicative $\dg$.

\section{Theorem~\ref{Tminors}: The three invariants for minors of
additive-multiplicative distance matrices}\label{S3}

In recent work~\cite{CK-tree}, in addition to computing $\det(\dg)$ and
$\cof(\dg)$ in terms of the strong blocks of $G$, we also focused on
trees, and computed these invariants for certain ``admissible'' minors of
the additive-multiplicative distance matrix.

The goal in this section is to show that these formulas extend to
similarly admissible minors for an arbitrary graph $G$ -- where we remove
rows and columns corresponding to nodes such that what remains is a
connected induced subgraph.
Moreover, we also compute the third invariant above -- namely,
$\kappa$ -- at these minors of $\dg$ and at arbitrary nodes $v_0$, and
show the same dependence purely on the block-data. We begin with
$\det(\cdot)$ and $\cof(\cdot)$, via the second main result of this
paper:

\begin{utheorem}\label{Tminors}
Suppose $G$ has strong blocks $\{ G_e : e \in  E \}$, vertex set $V$, and
block-datum $\mathcal{G} = \{ \mathcal{G}_e = (a_e, D^*_{G_e}) : e \in E
\}$. Let $I,J'$ denote subsets of nodes of $G$ satisfying:
(a)~$|I| = |J'| \leq |V|-3$;
(b)~every vertex in $I \setminus J'$ is connected to $V \setminus I$
through a unique cut-vertex in $V \setminus (I \cup J')$ -- and similarly
upon exchanging $I$ and $J'$;
(c)~the induced subgraphs on $V \setminus I,\ V \setminus J',\ V
\setminus (I \cap J')$ are strongly connected.
Now let $E_\circ := E_{(I \cap J')^c}$ index the strong blocks $G_e$ in
the induced directed subgraph on the vertices $V \setminus (I \cap J')$.

As an additional notation, given a $V \times V$ matrix $D$, let $D_{I|J'}$
the submatrix formed by removing the rows and columns labelled by $I,J'$
respectively. Then $\det (\dg + xJ)_{I|J'}$ depends on the block-data but
not on the block-structure:
\begin{align}\label{Eminor-detcof}
\det (\dg + xJ)_{I|J'} \
& = \ \ \begin{cases}
\prod_{e \in E_\circ} \kappa(\dge) \left[ x + \sum_{e \in
E_\circ} (a_e - x) \left( 1 - \frac{\cof(\dge)}{\kappa(\dge)}
\right) \right], & \text{if } |I \Delta J'| = 0,\\
\displaystyle \prod_{e \in E_\circ \setminus \{ \{ p(i_0), i_0 \}, \{
j_0, p(j_0) \} \}} \kappa(\dge) \cdot (-1)^{|V(G)|-1} a_{\{ p(i_0), i_0
\}}\\
\qquad \qquad \times (a_{\{ j_0, p(j_0) \} } - x)
(m_{(p(i_0), i_0)} - 1) (m_{(j_0, p(j_0))} - 1), \ & \text{if } |I
\Delta J'| = 2,\\
0, &\text{if } |I \Delta J'| > 2.
\end{cases}
\end{align}
Here, the denominators (for $I = J'$) are simply placeholders to cancel
with a factor in the numerators.
We also assume that if $|I \Delta J'| = 2$, then the nodes $i_0, j_0$ are
given by $I \setminus J' = \{ i_0 \}, \ J' \setminus I = \{ j_0 \}$.
In this case, $i_0, j_0$ are pendant in the induced subgraph on $V
\setminus (I \cap J')$, and $p(i_0), p(j_0)$ are their unique respective
neighbors in $V \setminus (I \cap J')$.
\end{utheorem}

\noindent We also compute the invariant $\kappa((\dg)_{I|J'}, v_0)$, in
Proposition~\ref{Pkappa}. This is new even for trees.

Notice that in addition to the original case of $\det(\dg + xJ)$ when $I
= J' = \emptyset$, Theorem~\ref{Tminors} subsumes
\cite[Theorem~A]{CK-tree}. The latter is because now $G$ is a tree, so
every $v \in I \setminus J'$ is indeed connected to $V \setminus I$
through a unique cut-vertex $v_0 \not\in I$. If $v_0 \in J'$, consider
the node $x \sim v_0$ that lies along the path from $v_0$ to $v$.
Disconnecting the edge $x - v_0$ separates $V$ into $I$ and $J'$, which
contradicts the hypotheses of \cite[Theorem~A]{CK-tree}.
(Similarly if $I$ and $J'$ are exchanged.)

\begin{remark}\label{Rdodgson}
Our recent work \cite[Theorem~A]{CK-tree} showed the special case for
trees. The $|I \Delta J'| = 2$ case in it was proved using multiple
ingredients:
(i)~the case of $|I \Delta J'| > 2$,
(ii)~the case of $|I \Delta J'| = 0$ (which is
precisely~\eqref{Emaster}),
(iii)~Dodgson condensation (i.e.~the Desnanot--Jacobi identity), and
(iv)~Zariski density.
Our alternate proof below for $|I \Delta J'| = 2$ and all graphs, avoids
using \textit{all} of these ingredients, and works using just the
definitions and the multiplicativity of $\kappa$. Thus it uses simpler
ingredients (modulo the longwinded computations proving~\eqref{Emghh1} in
Theorem~\ref{Tmasterghh}); it is conceptually more transparent even for
trees; and it applies more generally -- in fact to all graphs.
\end{remark}

Before proving the theorem, we mention another special case,
where $x = a_e = 1\ \forall e$:

\begin{cor}
Let $I \neq J'$ be as in Theorem~\ref{Tminors}. Then $\det(D^*_G)_{I|J'}
= 0$.
\end{cor}

\begin{proof}[Proof of Theorem~\ref{Tminors}]
Note that we may replace $G$ by the induced subgraph on the nodes in
$E_\circ$, i.e.~in $\bigcup_{e \in E_\circ} V(G_e)$. In other words, we
may assume $E = E_\circ$, or equivalently, that $I, J'$ are disjoint. Now
if $I = J'$ then the result reduces to Theorem~\ref{Tmasterghh}; see also
the discussion following its statement.

Henceforth assume $I \Delta J'$ is nonempty. The first claim is that
there exist $w \in J' \setminus I$ and $v \in V \setminus (I \cup J')$
such that $v \sim w$. For if not, consider any such $w,v$ for which $d(v,
w)$ is minimal -- here, $d(v,w)$ denotes the shortest path distance using
unweighted, bidirected edges. Then $v \sim w$, for if not then all paths
from $v$ to $w$ contain a node in $I$, contradicting the connectedness of
$V \setminus I$.

There are two other observations here: first, if $d(v,w)$ is minimal then
$v$ is a cut-vertex by assumption~(b). Thus $v$ disconnects the graph;
moreover, in the graph block corresponding to $w$, there cannot be a node
in $V \setminus J'$ by assumption~(b).

We next work out the (shorter) case of $|I \Delta J'| > 2$. There are two
sub-cases:

\begin{enumerate}
\item Suppose there are two such vertices $w_1 \neq w_2 \in J' \setminus
I$ that are adjacent to (cut) vertices $v_1, v_2 \in V \setminus (I \cup
J')$. Note here that $v_1$ may equal $v_2$. In this case, the $w_1, w_2,
v_1, v_2$ rows (with the columns for $J'$ removed) occur in $(\dg +
xJ)_{I|J'}$.

Given a cut-vertex $v$ and a node $w \neq v$, let $G_{v \to w}$ be the
unique maximum induced subgraph of $G$ that contains $v,w$ and for which
$v$ is not a cut-vertex. Then
(i)~the nodes in $G_{v_l \to w_l} \setminus \{ v_l \}$ for $l=1,2$ are
necessarily contained in $J'$ by the hypotheses, hence
(ii)~the corresponding columns are deleted from $(\dg + xJ)_{I|J'}$.

Now let the truncated row of $(\dg + xJ)_{I|J'}$ (with columns indexed by
$V \setminus J'$) corresponding to $v \in V$ be denoted by ${\bf d}^T_v =
(d(v, w) + x)_{w \in V}$.
Since ${\bf d}^T_{w_l}, {\bf d}^T_{v_l}$ indeed occur for $l=1,2$, the
determinant of $(\dg + xJ)_{I|J'}$ remains unchanged if one first
performs the row operations
\[
{\bf d}^T_{w_l} \mapsto {\bf d}^T_{w_l} - m_{w_l v_l} {\bf d}^T_{v_l},
\qquad l=1,2.
\]
But this gives a matrix with two proportional rows, since
\[
{\bf d}^T_{w_l} - m_{w_l v_l} {\bf d}^T_{v_l} = (a_{e_l} - x) (m_{w_l
v_l} - 1) {\bf e}^T, \qquad l=1,2,
\]
where $e_l$ is the unique block of $G$ containing the edge $\{ w_l, v_l
\}$. Hence $\det (\dg + xJ)_{I|J'} = 0$.

\item If the previous scenario does not occur, then there is a unique
cut-vertex $v_0 \in V \setminus (I \cup J')$ and a unique node $w_0 \in
J' \setminus I$ such that $w_0 \sim v_0$ and all other nodes in $J'$ are
connected to $v_0$ through $w_0$. Choose $w_1 \in J' \setminus (I \sqcup
\{ w_0 \})$ that is closest in the (unweighted, undirected) shortest path
distance to $w_0$. If $w_1$ is not adjacent to $w_0$, then
(i)~all nodes on all (unweighted, undirected) paths between $w_0 \in J'$
and $w_1 \in J'$ lie in $J'$, since $V \setminus J'$ is connected; and
(ii)~every path from $w_1$ to $w_0$ contains a node in $I$, and hence in
$I \cap J'$ by~(i). But then $w_0, w_1 \in J' \setminus I \subset V
\setminus I$ are separated by $I$, contradicting the hypotheses.

It follows that $w_1 \longleftrightarrow w_0 \longleftrightarrow v_0$.
Moreover, the columns corresponding to the nodes in $G_{v_0 \to w_0}
\setminus \{ v_0 \}$ are removed in $(\dg + xJ)_{I|J'}$, while the rows
corresponding to $w_1, w_0, v_0$ are all present. Hence as in the
previous case, the determinant of the minor $(\dg + xJ)_{I|J'}$ remains
unchanged upon carrying out -- in order -- the following row operations:
\[
{\bf d}^T_{w_1} \mapsto {\bf d}^T_{w_1} - m_{w_1 w_0} {\bf d}^T_{w_0},
\qquad
{\bf d}^T_{w_0} \mapsto {\bf d}^T_{w_0} - m_{w_0 v_0} {\bf d}^T_{v_0}.
\]
As above, this gives a matrix with two proportional rows, since
\[
{\bf d}^T_{w_1} - m_{w_1 w_0} {\bf d}^T_{w_0} = (a_{e_1} - x) (m_{w_1
w_0} - 1) {\bf e}^T, \qquad
{\bf d}^T_{w_0} - m_{w_0 v_0} {\bf d}^T_{v_0} = (a_{e_0} - x) (m_{w_0
v_0} - 1) {\bf e}^T,
\]
where $e_0, e_1$ are the blocks of $G$ containing the edges $\{ v_0, w_0
\}$ and $\{ w_0, w_1 \}$ respectively. It again follows that $\det (\dg +
xJ)_{I|J'} = 0$.
\end{enumerate}

The final case is when $| I \Delta J' | = 2$, so that we can assume
\[
I = I \setminus J' = \{ i_0 \}, \qquad
J' = J' \setminus I = \{ j_0 \}.
\]
Let $i_0, j_0$ be connected to $V \setminus I, V \setminus J'$ through
the cut-vertices $p(i_0), p(j_0)$ respectively. By the same reasoning as
above,
\[
\{ i_0, p(i_0) \} \subset G_{p(i_0) \to i_0} \subset \{ p(i_0) \} \sqcup
(I \setminus J') = \{ p(i_0) \} \sqcup I = \{ i_0, p(i_0) \},
\]
so $G_{p(i_0) \to i_0} = \{ i_0, p(i_0) \}$. The analogous results also
hold for $j_0$, so $i_0, j_0$ are pendant as claimed.

It remains to show~\eqref{Eminor-detcof}. Without loss of generality, let
\[
i_0 = 1, \quad j_0 = n, \quad V' := \{ 2, \dots, n-1 \} = V(G) \setminus
\{ i_0, j_0 \}.
\]
Also note from above that the edges $\{ i_0, p(i_0) \}, \{ j_0, p(j_0)
\}$ are strong blocks of $G$; we denote the corresponding additive
block-data by $a_1$ and $a_n$ respectively.

Using the definitions, it follows that
\[
\dg + xJ = \begin{pmatrix}
x & {\bf w}^T + x {\bf e}^T & d_{1,n} + x\\
{\bf d}(V', p(1)) + a_1 (m_{p(1),1} - 1) {\bf m}(V', p(1)) + x {\bf e} &
\dg|_{V' \times V'} + xJ & {\bf u} + x {\bf e}\\
d_{n,p(1)} + a_1 (m_{p(1),1}-1) m_{n,p(1)} + x & {\bf d}(n,V')^T + x {\bf
e}^T & x
\end{pmatrix}
\]
for a suitable choice of vectors ${\bf u}, {\bf w} \in R^{n-2}$. First
remove the first row and last column of this matrix. Next, the
determinant of the truncated matrix is unchanged upon subtracting the
$p(1)$th column from the first, so it suffices to compute the determinant
of
\[
D' := \begin{pmatrix}
a_1 (m_{p(1),1} - 1) {\bf m}(V', p(1)) & \dg|_{V' \times V'} + xJ\\
a_1 (m_{p(1),1}-1) m_{n,p(n)} m_{p(n), p(1)} & {\bf d}(n,V')^T + x {\bf
e}^T
\end{pmatrix}.
\]
To compute $\det(D')$, the factor $a_1 (m_{p(1),1}-1)$ can be taken out
of the first column. Next, since
\[
{\bf d}(n,V') = a_n (m_{n,p(n)} - 1) {\bf e} + m_{n,p(n)} {\bf
d}(p(n),V'),
\]
we subtract $m_{n,p(n)}$ times the $p(n)$th row of $D'$ from its final
row, to obtain:
\[
\det (\dg + xJ)_{1|n} = a_1 (m_{p(1),1}-1) \det \begin{pmatrix}
{\bf m}(V',p(1)) & \dg|_{V' \times V'} + xJ \\
0 & (a_n - x) (m_{n,p(n)}-1) {\bf e}^T
\end{pmatrix}
\]

We think of this $2 \times 2$ block matrix as relating to the graph $G'$
with two vertices and two edges removed, i.e., with nodes~$V'$ and strong
blocks precisely the strong blocks of $G$ with two fewer edge-blocks.
Now take the factor $(a_n - x) (m_{n,p(n)}-1)$ out of the last row, and
use the following identity~\eqref{Enew-identity}, which we isolate into a
standalone lemma of possible independent interest, to obtain:
\[
\det (\dg + xJ)_{1|n} = a_1 (m_{p(1),1}-1) \cdot 
(a_n - x) (m_{n,p(n)}-1) \cdot \kappa(\dg|_{V' \times V'}).
\]
But $G'$ has the same strong blocks as $G$ (with two fewer edges), so
by~\eqref{Emghh2} we obtain the desired result~\eqref{Eminor-detcof} for
$|I \Delta J'| = 2$.
\end{proof}

The following identity was used in the closing arguments above.

\begin{lemma}\label{Lminors}
Suppose $G$ is a graph with strong blocks $\{ G_e : e \in E \}$ and
block-datum $\mathcal{G} = \{ (a_e, D^*_{G_e}) \in R^{1 + |V|^2} : e
\in E \}$ as above. Fix $v \in V$, and say $x$ commutes with $R$. Then,
\begin{equation}\label{Enew-identity}
M(x) := \begin{pmatrix} {\bf m}(V, v) & \dg + xJ \\ 0 & {\bf e}^T
\end{pmatrix} \quad \implies \quad \det(M(x)) = (-1)^{|V|-1}
\kappa(\dg).
\end{equation}
\end{lemma}

\begin{proof}
Label the rows and columns of the matrix $M(x)$ by $(V; \infty)$ and
$(\infty; V)$ respectively. First pre-multiply the final row by $x {\bf
e}$ and subtract this from the upper $1 \times 2$ block-submatrix of
$M(x)$. This shows that $\det(M(x)) = \det(M(0))$, and so we work with
$x=0$ henceforth.

Subtract the $v$th column of $M(0)$ from all columns indexed by nodes in
$V \setminus \{ v \}$. Note that since $\dg$ has $(v,v)$-entry $0$,
this leaves unchanged the $v$th row of $M(0)$.
Call the new matrix $M'$, and subtract $m(w,v)$-times the $v$th row of
$M'$ from the $w$th row, for each $w \in V \setminus \{ v \}$. Now the
first column and the last row (both indexed by $\infty$) are standard
basis vectors. Expanding along the first column and the last row, and
using~\eqref{Ekappa}, we obtain $\det(M(x)) = (-1)^{|V|-1} \kappa(\dg)$,
as claimed.
\end{proof}

\subsection{Vanishing of $\kappa$ for minors}

Having computed $\det(\dg)_{I|J'}$ and $\cof(\dg)_{I|J'}$ in
Theorem~\ref{Tminors}, we complete the set by proving:

\begin{prop}\label{Pkappa}
Suppose $\mathcal{G}$ and $I, J' \subset V(G)$ are as in
Theorem~\ref{Tminors}. Let $v_0 \in V \setminus (I \cup J')$. Then:
\begin{equation}
\kappa((\dg)_{I|J'}, v_0) = \begin{cases}
\prod_{e \in E_\circ} \kappa(\dge), \qquad
&\text{if } I = J',\\
0 &\ \text{otherwise},
\end{cases}
\end{equation}
where $E_\circ = E_{(I \cap J')^c}$ is as in Theorem~\ref{Tminors}.
In particular, this too is independent of $v_0 \in V \setminus (I \cup
J')$.
\end{prop}

Notice from Definition~\ref{Dkappa} that $\kappa$ is not yet defined for
non-principal submatrices of $\dg$. We begin by doing so; now the final
part of Proposition~\ref{Pkappa} is in the spirit of
Theorem~\ref{Tmasterghh}. (In fact, Theorem~\ref{Tmasterghh} mostly
involved showing that $\kappa((\dg)_{I|J'}, v_0)$ is independent of
$v_0$, for $I = J' = \emptyset$.)

\begin{definition}
Suppose $G$ is a graph with additive-multiplicative block-datum
$\mathcal{G}$, and $I,J' \subset V(G)$ be equal-sized subsets. Let $v_0
\in V \setminus (I \cup J')$, and write
\[
(\dg)_{I|J'} = \begin{pmatrix} D_1 & {\bf u}_1 \\ {\bf w}_1^T & 0
\end{pmatrix},
\]
where $D_1 = (\dg)_{(I \cup \{ v_0 \})^c \times (J' \cup \{ v_0 \})^c}$.
Now define:
\begin{equation}\label{Ekappa2}
\kappa((\dg)_{I|J'}, v_0) := \det \left( D_1 - {\bf u}_1 \, {\bf
e}((J' \cup \{ v_0 \})^c)^T - {\bf m}((I \cup \{ v_0 \})^c, v_0) {\bf
w}_1^T \right).
\end{equation}
\end{definition}

Notice this generalizes the special case of $I=J'$ (i.e., $\kappa(\dg,
v_0)$) in Definition~\ref{Dkappa}; however, we did not mention this more
general case earlier because it is only applied (in this paper) in
proving Proposition~\ref{Pkappa}. Indeed, the more restrictive
Definition~\ref{Dkappa} is sufficient in order to show our three main
results: Theorems~\ref{Tmasterghh}, \ref{Tminors},
and~\ref{Tinverse-general}.

\begin{proof}[Proof of Proposition~\ref{Pkappa}]
The $I=J'$ case follows by Theorem~\ref{Tmasterghh}. Now suppose $I \neq
J'$; as in the proof of Theorem~\ref{Tminors}, we may assume without loss
of generality that $I, J'$ are disjoint. The first step is to
generalize Lemma~\ref{Lminors}:

\begin{lemma}
With notation as above, and for all $x$ commuting with $R$,
\begin{equation}\label{Ekappa3}
M(x) := \begin{pmatrix} {\bf m}(V \setminus I, v_0) & (\dg)_{I|J'} + xJ
\\ 0 & {\bf e}(V \setminus J')^T \end{pmatrix} \quad \implies \quad
\det(M(x)) = (-1)^{|V \setminus I|-1} \kappa((\dg)_{I|J'}, v_0).
\end{equation}
\end{lemma}

\noindent The proof is the same as that of Lemma~\ref{Lminors}, and is
hence omitted.

We now show the result for $I \neq J'$. By~\eqref{Ekappa3}, it suffices
to show that $M(0)$ is singular. As in the proof of
Theorem~\ref{Tminors}, choose $w \in J' \setminus I$ and a cut-vertex $v
\in V \setminus (I \cup J')$ which is adjacent to $w$. Now the columns
corresponding to the nodes of $G_{v \to w} \setminus \{ v \}$ are removed
from $(\dg)_{I|J'}$. Label every row of $M(0)$ -- except the last row --
by the corresponding node in $V \setminus I$, and subtract $m_{w,v}$
times the $v$th row from the $w$th row. Then the new $w$th row is $d(w,v)
\cdot (0, {\bf e}(V \setminus J')^T)$. As this is proportional to the
final row, the proof is complete.
\end{proof}

\section{Theorem~\ref{Tinverse-general}: The inverse matrix for
additive-multiplicative graphs}\label{S4}

In this section, we provide a closed-form expression for the inverse of
the distance matrix $\dg$ for an arbitrary graph $G$, in terms of the
additive-multiplicative block-datum $\mathcal{G}$ and the geometry
of~$G$. In particular, our result subsumes~\cite[Theorem B]{CK-tree} for
``additive-multiplicative trees'', which in turn implies all known
formulas for trees -- including by Graham--Lov\'asz, Bapat and his
coauthors, and Zhou--Ding (see the discussion and references around
\textit{loc.~cit.}).
More precisely, that result for trees involved explicit formulas, which
we now explain more conceptually (for general graphs).

A striking feature of our formula is that the matrix $\dg^{-1}$ is a
rank-one update not of the graph Laplacian $L_\mathcal{G}$ (as in
previous works~\cite{BKN,GL}) but of a different matrix $C_\mathcal{G}
(D^*_G)^{-1}$, where $C_\mathcal{G}$ was defined for trees
in~\cite{CK-tree}. That said, our formula specializes to all previous
versions in the literature.

\begin{definition}\label{Dinverse}
Given the block-datum $\mathcal{G} = \{ \mathcal{G}_e = ( a_e, D^*_{G_e}
) : e \in E \}$ for a graph $G$, with strong blocks $G_e$ and
multiplicative distance matrix $D^*_G$ that is invertible, define
\begin{equation}
\tauin^T := {\bf e}^T (D^*_G)^{-1}, \qquad
\tauout := (D^*_G)^{-1} {\bf e}.
\end{equation}
Given a node $i$ lying in a strong block $e \in E$, define $G_{i \to e}$
to be the unique maximum induced subgraph of $G$ that contains $e$ and
for which $i$ is not a cut-vertex. Now define for $i \in V(G)$:
\begin{equation}
\beta_i := \frac{\kappa(\dg)}{\det(\dg)} \sum_{e \in E : i \in e}
\frac{1}{a_e} \sum_{f \in E : f \subset G_{i \to e}}
\frac{\det(\dgf)}{\kappa(\dgf)},
\end{equation}
and the matrix $C_\mathcal{G} \in R^{V(G) \times V(G)}$ via:
\begin{equation}\label{ECmatrix}
(C_\mathcal{G})_{ij} := \begin{cases} \beta_i, & \text{if } j=i,\\
\beta_i - \frac{1}{a_e}, \qquad  & \text{if } j \neq i, \ j \in G_{i \to
e}.
\end{cases}
\end{equation}
\end{definition}

\begin{remark}\label{RCmatrix}
For every node $i \in V(G)$, the subgraphs $G_{i
\to e}$ partition the edges and the nodes:
\[
E(G) = \bigsqcup_{e \in E : i \in e}  E(G_{i \to e}), \qquad
V(G) = \{ i \} \sqcup \bigsqcup_{e \in E : i \in e} (V(G_{i \to e})
\setminus \{ i \}).
\]
Also note that if $i$ is not a cut-vertex, then $G_{i \to e} = G$, and
hence $\beta_i = \frac{1}{a_e}$ where $e \in E$ is the unique strong
block containing $i$. In particular, $C_\mathcal{G}$ has $i$th row
$\frac{1}{a_e} {\bf e}_i^T$.
\end{remark}

When $G$ is a tree with edge-datum $\mathcal{T}$,
Definition~\ref{Dinverse} specializes to the corresponding one
in~\cite{CK-tree} -- see~\eqref{Einverse-tree} below. That said, the
following formula for $\dg^{-1}$ is simpler than that for $\dt^{-1}$
in~\cite{CK-tree}:

\begin{utheorem}\label{Tinverse-general}
With the above notation -- in particular assuming that $\dg, D^*_G$ are
invertible,
\begin{equation}\label{Einverse}
\dg^{-1} = \frac{\kappa(\dg)}{\det(\dg)} \tauout \tauin^T + C_\mathcal{G}
(D^*_G)^{-1}.
\end{equation}
\end{utheorem}

Theorem~\ref{Tinverse-general} applies to all additive-multiplicative
matrices, of all graphs. As we show in Section~\ref{Shypertrees}, the
formula~\eqref{Einverse} indeed specializes to the
formula~\eqref{Einverse-tree} for $\dt^{-1}$ for trees, shown in
\cite[Theorem~B]{CK-tree}. This is in spite of the absence of a Laplacian
matrix as in the literature for trees~\cite{BKN,GL}. On a related note, in
Section~\ref{Shypertrees} we define such a Laplacian $L_\mathcal{G}$ for
general graphs, and it is closely related to the matrix $C_\mathcal{G}
(D^*_G)^{-1}$, which by Theorem~\ref{Tinverse-general} is the ``correct''
matrix to use for $\dg^{-1}$.

\begin{remark}
Akin to the invariants $\kappa(\dg), \det(\dg)$, $\cof(\dg)$, the inverse
$\dg^{-1}$ can also be recovered from the block-datum $\mathcal{G} = \{
(a_e, D^*_{G_e}) : e \in E \}$. This is because by \cite[Theorem
C]{CK-tree},
\begin{equation}\label{Eghh-inverse}
(D^*_G)^{-1} = \sum_{e \in E} [ (D^*_{G_e})^{-1} ]_{V(G_e)} + \Id_{V(G)}
- \sum_{e \in E} [ \Id_{G_e} ]_{V(G_e)},
\end{equation}
where given $S \subset V(G)$ and a matrix $A \in R^{|S| \times
|S|}$, we define $[A]_S$ to be the $|V(G)| \times |V(G)|$ matrix with
principal submatrix $A$ over the rows and columns indexed by $S$, and all
other entries zero.
\end{remark}

\begin{remark}
In the sequel, we will assume that $\dg$ and $\kappa(\dg)$ are
invertible, and hence so are all $a_e$ (if $|V(G_e)| > 1$) and
$D^*_{G_e}$ -- by Zariski density. For more on this, we refer the reader
to Remark~\ref{Rzariski-short} and to similar applications of Zariski
density in the special case in~\cite{CK-tree}.
\end{remark}

\begin{remark}\label{Rae-constant}
Suppose $a_e \equiv a\ \forall e \in E$, for some fixed scalar $a$. By
Remark~\ref{Rspecial}(2) or~\cite{CK-tree}, in this case $\dg = a(D^*_G -
J)$. Moreover, $C_\mathcal{G} = \frac{1}{a} {\rm Id}_{V(G)}$ by
Theorem~\ref{Tmasterghh}. This yields:
\begin{align*}
\frac{\kappa(\dg)}{\det(\dg)} \tauout \tauin^T + C_\mathcal{G}
(D^*_G)^{-1} = &\ a^{-1} \left( \frac{\det(D^*_G) -
\cof(D^*_G)}{\det(D^*_G)} \right)^{-1} \cdot (D^*_G)^{-1} {\bf e} \cdot
{\bf e}^T (D^*_G)^{-1} + a^{-1} (D^*_G)^{-1}\\
= &\ a^{-1} \left[ (D^*_G)^{-1} + \frac{(D^*_G)^{-1} {\bf e} \cdot {\bf e}^T
(D^*_G)^{-1}}{1 - {\bf e}^T (D^*_G)^{-1} {\bf e}} \right],
\end{align*}
which by the Sherman--Morrison formula equals $a^{-1} (D^*_G - {\bf e}
{\bf e}^T)^{-1} = \dg^{-1}$. This applies e.g.~to all $q$-distance
matrices $D_q(G)$ where $a_e = 1/(q-1)$, and to blocks (graphs with no
cut-vertices).
\end{remark}

We now turn to the proof of Theorem~\ref{Tinverse-general}, and begin by
isolating some preliminary identities that will be used below.

\begin{prop}\label{Pprelim}
Notation as above. Then:
\begin{align}
\tauin^T {\bf e} = {\bf e}^T \tauout = &\ \frac{\cof(D^*_G)}{\det(D^*_G)}
= \frac{\cof(\dg)}{\kappa(\dg)}; \label{Eprojesh22}\\
{\bf e}^T C_\mathcal{G} = &\ \frac{\kappa(\dg) - \cof(\dg)}{\det(\dg)}
\cdot {\bf e}^T. \label{Eprojesh23}
\end{align}
\end{prop}

\begin{proof}
The identity~\eqref{Eprojesh22} is immediate from the definitions and
Theorem~\ref{Tmasterghh}. We now show the equality in~\eqref{Eprojesh23}
of the $j$th coordinate on both sides, for each fixed $j \in V(G)$. The
left-side equals
\[
\sum_{i \in V(G)} \beta_i - \sum_{i \in V(G)} \frac{1}{a_{e_{i \to j}} },
\]
where $e_{i \to j}$ denotes the unique strong block $e \in E$ containing
$i$ such that $j \in G_{i \to e}$. We now convert the second sum, which
is over vertices, to one over blocks. It is not hard to see that every
block $e$ contains a unique node $i_0$ such that $e \neq e_{i_0 \to j}$;
hence the latter sum equals
\[
\sum_{e \in E} \frac{|V(G_e)| - 1}{a_e}.
\]

Next, we consider the former sum, again converting it into a sum over
blocks:
\[
\sum_{i \in V(G)} \beta_i = \frac{\kappa(\dg)}{\det(\dg)} \sum_{e \in E}
\frac{1}{a_e} \sum_{i \in V(G_e)} \sum_{f \in E : f \subset G_{i \to e}}
\frac{\det(\dgf)}{\kappa(\dgf)}
\]
For each fixed block $e \in E$, the inner double sum on the right is an
integer-linear combination of the ratios $\det(\dgf) / \kappa(\dgf)$ over
$f \in E$. Count the coefficient of this term, say $n_{e,f} \in
\Z$ for fixed $e$ and arbitrary $f \in E$. If $f=e$, then the ratio
$\det(\dge) / \kappa(\dge)$ occurs for every summand $i \in V(G_e)$ in
the inner double sum above, so $n_{e,e} = |V(G_e)|$. If instead $f \neq
e$, then the ratio $\det(\dgf) / \kappa(\dgf)$ occurs for every summand
$i \in V(G_e)$, except for the unique cut-vertex $i_0 \in V(G_e)$
separating $f$ from $G_{i_0 \to e}$ (which includes $e$). Hence, $n_{e,f}
= |V(G_e)| - 1$ if $f \neq e$. Therefore:
\begin{align*}
\sum_{i \in V(G)} \beta_i = &\ \frac{\kappa(\dg)}{\det(\dg)} \sum_{e \in
E} \frac{1}{a_e} \left( \frac{\det(\dge)}{\kappa(\dge)} + (|V(G_e)|
- 1) \sum_{f \in E} \frac{\det(\dgf)}{\kappa(\dgf)} \right)\\
= &\ \frac{\kappa(\dg)}{\det(\dg)} \sum_{e \in E} \frac{\det(\dge)}{a_e
\kappa(\dge)} + \sum_{e \in E} \frac{|V(G_e)| - 1}{a_e},
\end{align*}
where the final equality follows from~\eqref{Emghh3}.

Putting together these computations, for each $j \in V(G)$ we have
by~\eqref{Etemp} and Theorem~\ref{Tmasterghh}:
\[
({\bf e}^T C_\mathcal{G})_j
= \frac{\kappa(\dg)}{\det(\dg)} \sum_{e \in E} \frac{\det(\dge)}{a_e
\kappa(\dge)}
= \frac{\kappa(\dg)}{\det(\dg)} \sum_{e \in E} \left( 1-
\frac{\cof(\dge)}{\kappa(\dge)} \right)
= \frac{\kappa(\dg)}{\det(\dg)} \left( 1- \frac{\cof(\dg)}{\kappa(\dg)}
\right).
\]
Since this holds for every vertex $j \in V(G)$, the proof
of~\eqref{Eprojesh23} is complete.
\end{proof}

With Proposition~\ref{Pprelim} in hand, we show our final main result.

\begin{proof}[Proof of Theorem~\ref{Tinverse-general}]
The proof is by induction on the number of strong blocks $|E|$ of $G$,
with the case $|E|=1$ worked out in Remark~\ref{Rae-constant}. For the
induction step, assume that we know the result for $G$; now add to $G$
the pendant block $f$, separated from $G$ by the cut-vertex $v_0$. We
first set notation. Let $\overline{G}$ be the new graph, and write
$\dgbar$ for the additive-multiplicative matrix for $\overline{G}$. Thus:
\begin{align}\label{Edgbar}
\begin{aligned}
\dgbar := &\ \begin{pmatrix} \dg & H\\ K & a_f (D_2^* - J)
\end{pmatrix},\\
\text{where} \quad D^*_f = &\ \begin{pmatrix} 1 & {\bf w}^T \\ {\bf u} &
D_2^* \end{pmatrix}, \qquad
H := \dg {\bf e}_{v_0} {\bf e}^T + a_f D^*_G {\bf e}_{v_0} ({\bf w} -
{\bf e})^T,\\
K := &\ a_f ({\bf u} - {\bf e}) {\bf e}^T + {\bf u} {\bf
e}_{v_0}^T \dg.
\end{aligned}
\end{align}

In the sequel, we also use the formula for the inverse of a $2 \times 2$
square block matrix:
\begin{equation}\label{Eblock-inverse}
M = \begin{pmatrix} D_1 & H\\ K & D_2 \end{pmatrix} \quad \implies \quad
M^{-1} = \begin{pmatrix} D_1^{-1} + D_1^{-1} H \Psi^{-1} K D_1^{-1} & -
D_1^{-1} H \Psi^{-1} \\ - \Psi^{-1} K D_1^{-1} & \Psi^{-1} \end{pmatrix},
\end{equation}
where the $(1,1)$-block is assumed to be invertible, and $\Psi$ denotes
the Schur complement
\[
\Psi = D_2 - K D_1^{-1} H.
\]

\noindent The following special case is of interest:
\begin{equation}
(D^*_f)^{-1} = \begin{pmatrix} 1 + {\bf w}^T X^{-1} {\bf u} & - {\bf w}^T
X^{-1} \\ -X^{-1} {\bf u} & X^{-1} \end{pmatrix}, \qquad \text{where } X
:= D_2^* - {\bf u} {\bf w}^T.
\end{equation}

\noindent From this and~\eqref{Eghh-inverse}, one has the following
formula for $(D^*_{\overline{G}})^{-1}$:
\begin{align}\label{Emghh-inverse}
\begin{aligned}
(D^*_{\overline{G}})^{-1} = &\ [ (D^*_G)^{-1} ]_{V(G)} + [ (D^*_f)^{-1}
]_{V_f} - {\bf e}_{v_0} {\bf e}_{v_0}^T\\
= &\ \begin{pmatrix} (D^*_G)^{-1} + ({\bf w}^T X^{-1} {\bf u}) {\bf
e}_{v_0} {\bf e}_{v_0}^T & {\bf e}_{v_0} (-{\bf w}^T X^{-1}) \\ (- X^{-1}
{\bf u}) {\bf e}_{v_0}^T & X^{-1} \end{pmatrix}.
\end{aligned}
\end{align}

\noindent \textbf{Step 1:}
We break up the proof into steps for ease of exposition. Let $D_f$ denote
the additive-multiplicative distance matrix of $f$:
\begin{equation}
D_f := a_f (D^*_f - J) = a_f \begin{pmatrix}
0 & ({\bf w} - {\bf e})^T \\
{\bf u} - {\bf e} & D_2^* - J
\end{pmatrix}.
\end{equation}
Then we claim:
\begin{equation}\label{Eidentity}
- ({\bf w} - {\bf e})^T X^{-1} ({\bf u} - {\bf e}) = \frac{\det(D_f)}{a_f
\kappa(D_f)}.
\end{equation}

Indeed, we compute using a row and a column operation and the above
formulas in this proof:
\begin{align*}
\det(D_f) = &\ \det a_f \begin{pmatrix} 
0 & ({\bf w} - {\bf e})^T\\
{\bf u} - {\bf e} & X
\end{pmatrix}
= a_f^{|V_f|} \det(X) \left( - ({\bf w} - {\bf e})^T X^{-1} ({\bf u} -
{\bf e}) \right)\\
= &\ a_f^{|V_f|} \det(D^*_f) \left( - ({\bf w} - {\bf e})^T X^{-1} ({\bf
u} - {\bf e}) \right),
\end{align*}
from which the claim follows. This identity will be used repeatedly in
our computations below.\medskip

\noindent \textbf{Step 2:}
We now explain our strategy. The left-hand side of~\eqref{Einverse},
namely~$(\dgbar)^{-1}$ where $\dgbar$ is given by~\eqref{Edgbar}, can be
computed using the formula~\eqref{Eblock-inverse}. On the other hand, the
right-hand side of~\eqref{Einverse} can be explicitly written out in $2
\times 2$ block form as well. We will carry out both of these steps and
show the equality of the two sides of~\eqref{Einverse}, block by block.

We begin by writing out analogues for $\dgbar$ of the vectors $\tauin,
\tauout$ for $\dg$; we denote these analogues by $\overline{\tauin},
\overline{\tauout}$ respectively:
\begin{equation}
\overline{\tauin} = \begin{pmatrix} \tauin^1 \\ \tauin^2 \end{pmatrix},
\qquad
\overline{\tauout} = \begin{pmatrix} \tauout^1 \\ \tauout^2 \end{pmatrix},
\end{equation}
where we have by~\eqref{Emghh-inverse}:
\begin{align}\label{Etauinout12}.
\begin{aligned}
(\tauin^1)^T = &\ \tauin^T + ({\bf w} - {\bf e})^T X^{-1} {\bf u} \cdot
{\bf e}_{v_0}^T, \qquad 
(\tauin^2)^T = - ({\bf w} - {\bf e})^T X^{-1},\\
\tauout^1 = &\ \tauout + {\bf w}^T X^{-1} ({\bf u} - {\bf e}) \cdot
{\bf e}_{v_0}, \qquad
\tauout^2 = - X^{-1} ({\bf u} - {\bf e}).
\end{aligned}
\end{align}

Also note that $C_{\overline{\mathcal{G} }}$ is block upper-triangular,
with $(2,1)$-block zero, and $(2,2)$-block $a_f^{-1} \Id$, since $V_f
\setminus \{ v_0 \}$ contains no cut-vertices. Writing
$C_{\overline{\mathcal{G}} } = \begin{pmatrix} \overline{C}_{11} &
\overline{C}_{12} \\ 0 & a_f^{-1} \Id \end{pmatrix}$,
the theorem reduces by~\eqref{Eblock-inverse} to showing:
\begin{align}\label{Einverse-block}
\begin{aligned}
&\ \begin{pmatrix} \dg^{-1} + \dg^{-1} H \Psi^{-1} K
\dg^{-1} & - \dg^{-1} H \Psi^{-1} \\ - \Psi^{-1} K \dg^{-1} & \Psi^{-1}
\end{pmatrix}\\
= &\ \frac{\kappa(\dgbar)}{\det(\dgbar)} \begin{pmatrix}
\tauout^1 (\tauin^1)^T &
\tauout^1 (\tauin^2)^T \\
\tauout^2 (\tauin^1)^T &
\tauout^2 (\tauin^2)^T
\end{pmatrix}\\
+ &\ \begin{pmatrix} \overline{C}_{11} & \overline{C}_{12} \\ 0 &
a_f^{-1} \Id \end{pmatrix}
\begin{pmatrix} (D^*_G)^{-1} + ({\bf w}^T X^{-1} {\bf u}) {\bf e}_{v_0}
{\bf e}_{v_0}^T & {\bf e}_{v_0} (-{\bf w}^T X^{-1}) \\ (- X^{-1} {\bf u})
{\bf e}_{v_0}^T & X^{-1} \end{pmatrix}.
\end{aligned}
\end{align}

Finally, we state a useful consequence of Proposition~\ref{Pprelim}
and~\eqref{Einverse} (via the induction hypothesis):
\begin{equation}\label{Eidentity2}
{\bf e}^T \dg^{-1} = \frac{\kappa(\dg)}{\det(\dg)} \tauin^T.
\end{equation}

\noindent \textbf{Step 3:}
We now compute $\Psi$ using~\eqref{Eidentity2}, and equate $\Psi^{-1}$ to
the $(2,2)$-block on the right in~\eqref{Einverse-block}:
\begin{align*}
\Psi = &\ a_f (D_2^* - J) - K \dg^{-1} H\\
= &\ a_f (D_2^* - J) - 
\left( a_f ({\bf u} - {\bf e}) {\bf e}^T + {\bf u} {\bf e}_{v_0}^T \dg
\right) \dg^{-1} \left( \dg {\bf e}_{v_0} {\bf e}^T + a_f D^*_G {\bf
e}_{v_0} ({\bf w} - {\bf e})^T \right)\\
= &\ a_f (D_2^* - J) - \left( a_f ({\bf u} - {\bf e}) {\bf e}^T + a_f
{\bf u} ({\bf w} - {\bf e})^T + a_f^2 \frac{\kappa(\dg)}{\det(\dg)} ({\bf
u} - {\bf e}) ({\bf w} - {\bf e})^T \right)\\
= &\ a_f \left( X - a_f \frac{\kappa(\dg)}{\det(\dg)} ({\bf u} - {\bf e})
({\bf w} - {\bf e})^T \right).
\end{align*}
By the Sherman--Morrison formula and~\eqref{Eidentity},
\begin{align}\label{Epsiinverse}
\begin{aligned}
\Psi^{-1} = &\ a_f^{-1} \left[ X^{-1} + \frac{\displaystyle a_f
\frac{\kappa(\dg)}{\det(\dg)} X^{-1} ({\bf u} - {\bf e}) ({\bf w} - {\bf
e})^T X^{-1}}{\displaystyle 1 - a_f \frac{\kappa(\dg)}{\det(\dg)} ({\bf
w} - {\bf e})^T X^{-1} ({\bf u} - {\bf e})} \right]\\
= &\ a_f^{-1} X^{-1} + a_f^{-1} \frac{X^{-1} ({\bf u} - {\bf e}) ({\bf w}
- {\bf e})^T X^{-1}}{\displaystyle a_f^{-1} \frac{\det(\dg)}{\kappa(\dg)}
+ a_f^{-1} \frac{\det(D_f)}{\kappa(D_f)}}\\
= &\ a_f^{-1} X^{-1} + \frac{\tauout^2
(\tauin^2)^T}{\det(\dgbar)/\kappa(\dgbar)},
\end{aligned}
\end{align}
where the final equality uses~\eqref{Etauinout12} and the Master
GHH-formula~\eqref{Emghh3}. Note, this is the computation of the
$(2,2)$-block in the left-hand side of~\eqref{Einverse-block}. But it
also clearly equals the $(2,2)$-block in that right-hand side, as
desired.\medskip

\noindent \textbf{Step 4:}
We next check the equality of the $(2,1)$-blocks
in~\eqref{Einverse-block}. Using~\eqref{Eblock-inverse},
\eqref{Eidentity}, and~\eqref{Eidentity2}:
\begin{align*}
&\ -\Psi^{-1} K \dg^{-1}\\
= &\ - \Psi^{-1} \left( a_f \frac{\kappa(\dg)}{\det(\dg)} ({\bf u} - {\bf
e}) \tauin^T + {\bf u} {\bf e}_{v_0}^T \right)\\
= &\ \frac{-\kappa(\dg)}{\det(\dg)} \left[ -\tauout^2 + a_f
\frac{\kappa(\dgbar)}{\det(\dgbar)} \tauout^2 \cdot \frac{\det(D_f)}{a_f
\kappa(D_f)} \right] \tauin^T +
\frac{\kappa(\dgbar)}{\det(\dgbar)} \cdot ({\bf w} - {\bf e})^T X^{-1}
{\bf u} \cdot \tauout^2 {\bf e}_{v_0}^T\\
&\ - a_f^{-1} X^{-1} {\bf u} {\bf e}_{v_0}^T\\
= &\ \frac{\kappa(\dgbar)}{\det(\dgbar)} \left[ \tauout^2 \tauin^T +
({\bf w} - {\bf e})^T X^{-1} {\bf u} \cdot \tauout^2 {\bf e}_{v_0}^T
\right] - a_f^{-1} X^{-1} {\bf u} {\bf e}_{v_0}^T\\
= &\ \frac{\kappa(\dgbar)}{\det(\dgbar)} \tauout^2 (\tauin^1)^T -
a_f^{-1} X^{-1} {\bf u} {\bf e}_{v_0}^T,
\end{align*}
where the penultimate equality uses the Master
GHH-formula~\eqref{Emghh3}, and the final equality
uses~\eqref{Etauinout12}. But this is, once again, easily seen to equal
the $(2,1)$-block of the right-hand side
of~\eqref{Einverse-block}.\medskip

\noindent \textbf{Step 5:}
We now examine the $(1,2)$-blocks in~\eqref{Einverse-block}. Using the
induction hypothesis for~\eqref{Einverse}, the left-hand side yields:
\[
-\dg^{-1} H \Psi^{-1}
= - \left[ {\bf e}_{v_0} {\bf e}^T + a_f \left(
\frac{\kappa(\dg)}{\det(\dg)} \tauout + C_{\mathcal{G}} {\bf e}_{v_0}
\right) ({\bf w} - {\bf e})^T \right] \Psi^{-1}.
\]
By~\eqref{Eidentity}, the first term of this product equals
\begin{align*}
&\ - a_f^{-1} {\bf e}_{v_0} {\bf e}^T X^{-1} +
\frac{\kappa(\dgbar)}{\det(\dgbar)} \left[ {\bf w}^T X^{-1} ({\bf u} -
{\bf e}) + \frac{\det(D_f)}{a_f \kappa(D_f)} \right] {\bf e}_{v_0}
(\tauin^2)^T.
\end{align*}

\noindent Moreover, a parallel computation to the previous step yields:
\[
a_f ({\bf w} - {\bf e})^T \Psi^{-1} = -
\frac{\kappa(\dgbar)}{\det(\dgbar)} \frac{\det(\dg)}{\kappa(\dg)}
(\tauin^2)^T.
\]

Hence the $(1,2)$-block $-\dg^{-1} H \Psi^{-1}$ on the left
of~\eqref{Einverse-block} equals:
\begin{align}\label{Etemp1}
= &\ - a_f^{-1} {\bf e}_{v_0} {\bf e}^T X^{-1} +
\frac{\kappa(\dgbar)}{\det(\dgbar)} \cdot {\bf w}^T X^{-1} ({\bf u} -
{\bf e}) \cdot {\bf e}_{v_0} (\tauin^2)^T +
\frac{\kappa(\dgbar)}{\det(\dgbar)} \cdot \frac{\det(D_f)}{a_f
\kappa(D_f)} \cdot {\bf e}_{v_0} (\tauin^2)^T \notag\\
&\ + \frac{\kappa(\dgbar)}{\det(\dgbar)} \tauout (\tauin^2)^T +
\frac{\kappa(\dgbar)}{\det(\dgbar)} \frac{\det(\dg)}{\kappa(\dg)}
C_\mathcal{G} {\bf e}_{v_0} (\tauin^2)^T.
\end{align}

\noindent The first and third terms on the right-hand side
of~\eqref{Etemp1} add up to give
\begin{equation}\label{Etemp2}
- a_f^{-1} \frac{\kappa(\dgbar)}{\det(\dgbar)}
\frac{\det(\dg)}{\kappa(\dg)} {\bf e}_{v_0} {\bf e}^T X^{-1}
- a_f^{-1} \frac{\kappa(\dgbar)}{\det(\dgbar)}
\frac{\det(D_f)}{\kappa(D_f)} {\bf e}_{v_0} {\bf w}^T X^{-1},
\end{equation}

\noindent while combining the second and fourth terms of~\eqref{Etemp1}
yields
$\displaystyle \frac{\kappa(\dgbar)}{\det(\dgbar)} \tauout^1
(\tauin^2)^T$
via~\eqref{Etauinout12}. Now split the final term in~\eqref{Etemp1} using
\[
(\tauin^2)^T = {\bf e}^T X^{-1} - {\bf w}^T X^{-1}
\]
and pair these two terms with the two terms in~\eqref{Etemp2} to obtain:
\begin{align*}
-\dg^{-1} H \Psi^{-1}
= &\ \frac{\kappa(\dgbar)}{\det(\dgbar)} \tauout^1 (\tauin^2)^T +
\frac{\kappa(\dgbar)}{\det(\dgbar)} \frac{\det(\dg)}{\kappa(\dg)} \left(
- a_f^{-1} {\bf e}_{v_0} + C_\mathcal{G} {\bf e}_{v_0} \right) {\bf e}^T
X^{-1}\\
&\ + \frac{\kappa(\dgbar)}{\det(\dgbar)} \left( a_f^{-1}
\frac{\det(D_f)}{\kappa(D_f)} {\bf e}_{v_0} +
\frac{\det(\dg)}{\kappa(\dg)} C_\mathcal{G} {\bf e}_{v_0} \right) (-{\bf
w}^T X^{-1}).
\end{align*}

This is the $(1,2)$-block on the left-hand side
of~\eqref{Einverse-block}. The first term exactly matches the first term
on the right-hand side of~\eqref{Einverse-block}. Moreover, a careful
computation reveals that
\begin{align}\label{Ecareful}
\begin{aligned}
\overline{C}_{11} {\bf e}_{v_0} = &\ \frac{\kappa(\dgbar)}{\det(\dgbar)}
\left( a_f^{-1} \frac{\det(D_f)}{\kappa(D_f)} {\bf e}_{v_0} +
\frac{\det(\dg)}{\kappa(\dg)} C_\mathcal{G} {\bf e}_{v_0} \right),\\
\overline{C}_{12} = &\ \frac{\kappa(\dgbar)}{\det(\dgbar)}
\frac{\det(\dg)}{\kappa(\dg)} \left( - a_f^{-1} {\bf e}_{v_0} +
C_\mathcal{G} {\bf e}_{v_0} \right) {\bf e}^T.
\end{aligned}
\end{align}
From this and the above computations, the $(1,2)$-blocks
in~\eqref{Einverse-block} agree.\medskip

\noindent \textbf{Step 6:}
Finally, we reconcile the $(1,1)$-blocks in~\eqref{Einverse-block}. On the
left-hand side is $\dg^{-1} + \dg^{-1} H \cdot \Psi^{-1} K \dg^{-1}$. By
the induction hypothesis, Step 4, and the first line in Step 5, we have
\begin{align*}
&\ \dg^{-1} + \dg^{-1} H \cdot \Psi^{-1} K \dg^{-1}\\
= &\ \frac{\kappa(\dg)}{\det(\dg)} \tauout \tauin^T + C_\mathcal{G}
(D^*_G)^{-1} + \left[ a_f \left( \frac{\kappa(\dg)}{\det(\dg)} \tauout +
C_{\mathcal{G}} {\bf e}_{v_0} \right) ({\bf w} - {\bf e})^T + {\bf
e}_{v_0} {\bf e}^T \right]
\times\\
&\ \qquad \qquad \times \left( \frac{\kappa(\dgbar)}{\det(\dgbar)} \left[
X^{-1} ({\bf u} - {\bf e}) \tauin^T + ({\bf w} - {\bf e})^T X^{-1} {\bf
u} \cdot X^{-1} ({\bf u} - {\bf e}) {\bf e}_{v_0}^T \right] + a_f^{-1}
X^{-1} {\bf u} {\bf e}_{v_0}^T \right)\\
= &\ \frac{\kappa(\dg)}{\det(\dg)} \tauout \tauin^T + C_\mathcal{G}
(D^*_G)^{-1} + S_1 + \cdots + S_7,
\end{align*}
where
\begin{align}\label{S1to7}
\begin{aligned}
S_1 := &\ - \frac{\kappa(\dg)}{\det(\dg)}
\frac{\kappa(\dgbar)}{\det(\dgbar)} \frac{\det(D_f)}{\kappa(D_f)} \tauout
\tauin^T,\\
S_2 := &\ \frac{\kappa(\dgbar)}{\det(\dgbar)} ({\bf w} - {\bf e})^T
X^{-1} {\bf u} \cdot \tauout {\bf e}_{v_0}^T,\\
S_3 := &\ - \frac{\kappa(\dgbar)}{\det(\dgbar)}
\frac{\det(D_f)}{\kappa(D_f)} C_\mathcal{G} {\bf e}_{v_0} \tauin^T,\\
S_4 := &\ \frac{\kappa(\dgbar)}{\det(\dgbar)}
\frac{\det(\dg)}{\kappa(\dg)} ({\bf w}^T X^{-1} {\bf u} - {\bf e}^T
X^{-1} {\bf u}) \cdot C_{\mathcal{G}} {\bf e}_{v_0} {\bf e}_{v_0}^T,\\
S_5 := &\ \frac{\kappa(\dgbar)}{\det(\dgbar)} {\bf e}^T
X^{-1} ({\bf u} - {\bf e}) \cdot {\bf e}_{v_0} \tauin^T,\\
S_6 := &\ \frac{\kappa(\dgbar)}{\det(\dgbar)} {\bf e}^T X^{-1} ({\bf u} -
{\bf e}) \cdot ({\bf w} - {\bf e})^T X^{-1} {\bf u} \cdot {\bf e}_{v_0}
{\bf e}_{v_0}^T,\\
S_7 := &\ a_f^{-1} {\bf e}^T X^{-1} {\bf u} \cdot {\bf e}_{v_0} {\bf
e}_{v_0}^T.
\end{aligned}
\end{align}

We now start to combine these nine terms. The first term added to
$S_1$ and $S_2$ yields via~\eqref{Etauinout12}:
\begin{align*}
S_8 := &\ \frac{\kappa(\dgbar)}{\det(\dgbar)} \tauout^1 (\tauin^1)^T -
\frac{\kappa(\dgbar)}{\det(\dgbar)} {\bf w}^T X^{-1} ({\bf u} - {\bf e})
\cdot {\bf e}_{v_0} \tauin^T\\
&\ - \frac{\kappa(\dgbar)}{\det(\dgbar)} {\bf w}^T X^{-1} ({\bf u} - {\bf
e}) \cdot ({\bf w} - {\bf e})^T X^{-1} {\bf u} \cdot {\bf e}_{v_0} {\bf
e}_{v_0}^T.
\end{align*}

The second term above, combined with $S_3, S_5$, and the second sub-term
in $S_8$, yields:
\begin{align*}
S_9 := &\ \left( C_\mathcal{G} - \frac{\kappa(\dgbar)}{\det(\dgbar)}
\frac{\det(D_f)}{\kappa(D_f)} C_\mathcal{G} {\bf e}_{v_0} {\bf e}^T +
a_f^{-1} \frac{\kappa(\dgbar)}{\det(\dgbar)}
\frac{\det(D_f)}{\kappa(D_f)} {\bf e}_{v_0} {\bf e}^T \right)
(D^*_G)^{-1}.
\end{align*}
But the factor on the right preceding $(D^*_G)^{-1}$ is precisely
$\overline{C}_{11}$ from~\eqref{Einverse-block}, as a careful
verification reveals. Finally, $S_6$ and the third sub-term of $S_8$ add
up to yield via~\eqref{Eidentity}:
\[
S_{10} := \frac{\kappa(\dgbar)}{\det(\dgbar)} \frac{\det(D_f)}{a_f
\kappa(D_f)} ({\bf w}^T X^{-1} {\bf u} - {\bf e}^T X^{-1} {\bf u}) {\bf
e}_{v_0} {\bf e}_{v_0}^T.
\]

Combining all of these, the $(1,1)$-block on the left-hand side
of~\eqref{Einverse-block} equals
\begin{align*}
&\ \dg^{-1} + \dg^{-1} H \cdot \Psi^{-1} K \dg^{-1}\\
= &\ \frac{\kappa(\dgbar)}{\det(\dgbar)} \tauout^1 (\tauin^1)^T +
\overline{C}_{11} (D^*_G)^{-1} + S_4 + S_7 + S_{10}.
\end{align*}

Examining the $(1,1)$-blocks of~\eqref{Einverse-block}, it remains to
show that
\[
S_4 + S_7 + S_{10} = \overline{C}_{11} \cdot ({\bf w}^T X^{-1} {\bf u})
{\bf e}_{v_0} {\bf e}_{v_0}^T + \overline{C}_{12} (-X^{-1} {\bf u}) {\bf
e}_{v_0}^T.
\]
Notice from above that $S_4, S_7, S_{10}$ are all combinations of
\[
{\bf w}^T X^{-1} {\bf u} \cdot {\bf e}_{v_0} {\bf e}_{v_0}^T, \qquad
{\bf e}^T X^{-1} {\bf u} \cdot {\bf e}_{v_0} {\bf e}_{v_0}^T.
\]
Regroup $S_4 + S_7 + S_{10}$ in terms of these, and remove the ${\bf
e}_{v_0}^T$ from both sides of the preceding equation. Thus, it suffices
to show:
\begin{align}\label{Etoshow}
\begin{aligned}
&\ {\bf w}^T X^{-1} {\bf u} \cdot \frac{\kappa(\dgbar)}{\det(\dgbar)}
\left( a_f^{-1} \frac{\det(D_f)}{\kappa(D_f)} {\bf e}_{v_0} +
\frac{\det(\dg)}{\kappa(\dg)} C_\mathcal{G} {\bf e}_{v_0} \right)\\
+ &\ {\bf e}^T X^{-1} {\bf u} \cdot \frac{\kappa(\dgbar)}{\det(\dgbar)}
\left( a_f^{-1} \frac{\det(\dgbar)}{\kappa(\dgbar)} {\bf e}_{v_0} -
a_f^{-1} \frac{\det(D_f)}{\kappa(D_f)} {\bf e}_{v_0} -
\frac{\det(\dg)}{\kappa(\dg)} C_\mathcal{G} {\bf e}_{v_0} \right)\\
= &\ {\bf w}^T X^{-1} {\bf u} \cdot \overline{C}_{11} {\bf e}_{v_0} +
\overline{C}_{12} (-X^{-1} {\bf u}).
\end{aligned}
\end{align}
But this follows by applying~\eqref{Emghh3} and~\eqref{Ecareful}.
\end{proof}

\subsection{Special case 1: additive-multiplicative Laplacian, and
additive-multiplicative (hyper)trees}\label{Shypertrees}

We begin this final subsection by recalling the formula for $\dt^{-1}$
for trees equipped with a general additive-multiplicative datum $\{ (a_e,
m_e, m'_e) : e \in E \}$. Specifically, in~\cite{CK-tree} we defined
certain vectors $\tauin, \tauout$, a scalar $\at$, the graph Laplacian
matrix $L_\mathcal{T}$, and a matrix $C_\mathcal{T}$, such that
\begin{equation}\label{Einverse-tree}
\dt^{-1} = \frac{1}{\at} \tauout \tauin^T - L_\mathcal{T} + C_\mathcal{T}
\diag(\tauin).
\end{equation}

We now show that Theorem~\ref{Tinverse-general} can be recast into a
similar formula for $\dg^{-1}$, for any graph $G$. We begin by defining
$\alpha_\mathcal{G}$ and $L_\mathcal{G}$ in general -- the symbols
$\tauin, \tauout, C_\mathcal{G}$ were already defined above.

\begin{definition}
Given the block-datum $\mathcal{G} = \{ \mathcal{G}_e = ( a_e, D^*_{G_e})
: e \in E \}$ for a graph $G$, with strong blocks $G_e$ and invertible
$a_e, D^*_{G_e}$, define
\[
\alpha_{\mathcal{G}} := \frac{\det(\dg)}{\kappa(\dg)},
\]
and define the \textit{additive-multiplicative Laplacian} to be the
$|V(G)| \times |V(G)|$ matrix $L_{\mathcal{G}} = (l_{ij})$, with:
\[
l_{ij} := \begin{cases}
\frac{-1}{a_e}((D^*_G)^{-1})_{ij}, & \text{if } i \neq j \text{ in a
block } e \in E,\\
0 & \text{if } i,j \text{ lie in different blocks},\\
-\sum_{k \neq j} l_{kj}, \qquad & \text{if } i=j.
\end{cases}
\]
\end{definition}

Note that $L_\mathcal{G}$ has column sums zero; in the
special case of trees, for $q$- and classical distance matrices which
are symmetric, $L_\mathcal{G}$ was symmetric and had zero row sums,
e.g.~in~\cite{GL}. Moreover, it follows by~\eqref{Eghh-inverse} that
$D^*_G$ is a sum of block diagonal matrices (overlapping at the
cut-vertex diagonal entries), hence so is the additive-multiplicative
Laplacian matrix $L_{\mathcal{G}}$.

\begin{prop}\label{Ptreelike}
Notation as above. Then:
\begin{equation}\label{Eprojesh24}
C_\mathcal{G} = (-L_\mathcal{G} + C_\mathcal{G} \diag(\tauin)) D^*_G.
\end{equation}
In particular, and parallel to~\eqref{Einverse-tree}, one has:
\[
\dg^{-1} = \frac{1}{\alpha_{\mathcal{G} }} \tauout \tauin^T -
L_\mathcal{G} + C_\mathcal{G} \diag(\tauin).
\]
\end{prop}

In particular, this implies~\eqref{Einverse-tree}.
A related observation is that in~\cite{CK-tree}, we showed for trees via
explicit computations:
\[
\tauin^T {\bf m}_{\bullet \to l} = 1, \qquad \forall l \in V(G),
\]
where ${\bf m}_{\bullet \to l} := D^*_G {\bf e}_l$ denotes the vector
$(m_{vl})_{v \in V(G)}$. With our newfound understanding of $\tauin$,
this is now obvious, and for any graph $G$.

\begin{proof}
Define the ``block-degree'' of a vertex $v \in V(G)$ as follows:
\begin{equation}
d_E(v) := \# \{ e \in E : v \in e \}, \qquad V^{cut} :=
d_E^{-1}([2,\infty)).
\end{equation}
Notice that $V^{cut}$ is precisely the set of cut-vertices; in the case
of a tree, these are precisely the non-pendant nodes and $d_E(v)$ is
the degree of $v$.\footnote{As an aside, this suggests a different
definition/interpretation of pendant nodes in a graph: those which lie in
a unique strong block. This differs from the usual notion of a pendant
node -- i.e.~a node with ``usual'' degree one -- and we do not use this
alternate interpretation further.}

Once again assume (by Zariski density) that $\kappa(\dg)$ is invertible,
and hence so are all $a_e$ and $D^*_{G_e}$. Also define for convenience
$V_e := V(G_e)$ for each block $e \in E$. Then it follows
from~\eqref{Eghh-inverse},
the lines preceding Remark~\ref{Rae-constant}, and the definitions that
\begin{align}\label{Eprelim}
\begin{aligned}
C_\mathcal{G} = &\ \sum_{e \in E} \left[ \frac{1}{a_e} \Id_{V_e}
\right]_{V_e} +
\sum_{e \in E} \sum_{v \in e \cap V^{cut}} \left[ \left( \beta_v -
\frac{1}{a_e} \right) {\bf e}_v {\bf e}(V(G_{v \to e}))^T \right]_{V(G_{v
\to e})}\\
&\ + \sum_{v \in V^{cut}} (1 - d_E(v)) \beta_v {\bf e}_v {\bf e}_v^T,\\
(D^*_G)^{-1} = &\ \sum_{e \in E} [ (D^*_{G_e})^{-1} ]_{V_e} + \sum_{v \in
V^{cut}} (1 - d_E(v)) {\bf e}_v {\bf e}_v^T\\
-L_\mathcal{G} = &\ \sum_{e \in E} \frac{1}{a_e} \left[ (D^*_{G_e})^{-1}
- \diag(\tauin^e) \right]_{V_e},\\
\diag(\tauin) = &\ \diag( {\bf e}^T (D^*_G)^{-1}) = \sum_{e \in E} \left[
\diag(\tauin^e) \right]_{V_e} + \sum_{v \in V^{cut}} (1 - d_E(v)) {\bf
e}_v {\bf e}_v^T,
\end{aligned}
\end{align}
where $(\tauin^e)^T := {\bf e}(V_e)^T (D^*_{G_e})^{-1}$. (The fourth of
these formulas follows from the second.)

Notice that the final assertion in the proposition follows
from~\eqref{Eprojesh24} via Theorem~\ref{Tinverse-general}. Now,
showing~\eqref{Eprojesh24} is equivalent -- by Zariski density -- to
showing:
\[
C_\mathcal{G} ((D^*_G)^{-1} - \diag(\tauin)) + L_\mathcal{G} = 0.
\]
Using the formulas~\eqref{Eprelim}, we see that the terms in
$L_\mathcal{G}$ cancel out some of the terms in the first summand of
$C_\mathcal{G}$ times $((D^*_G)^{-1} - \diag(\tauin))$. Now writing
\[
M_f := \left[ (D^*_{G_f})^{-1} - \diag(\tauin^f) \right]_{V_f}, \qquad f
\in E
\]

\noindent it follows that
\begin{align}
\begin{aligned}\label{Erightside}
C_\mathcal{G} ((D^*_G)^{-1} - \diag(\tauin)) + L_\mathcal{G}
= &\ \sum_{e \in E} \sum_{f \in E, f \neq e} \frac{1}{a_e}
[\Id_{V_e}]_{V_e} M_f + \sum_{v \in V^{cut}, f \in E} (1 - d_E(v))
\beta_v {\bf e}_v {\bf e}_v^T M_f\\
&\ \ + \sum_{e,f \in E} \sum_{v \in e \cap V^{cut}} \left[ \left( \beta_v -
\frac{1}{a_e} \right) {\bf e}_v {\bf e}(V(G_{v \to e}))^T \right]_{V(G_{v
\to e})} M_f.
\end{aligned}
\end{align}

Denote the three sums on the right-side of~\eqref{Erightside} by $S_1,
S_2, S_3$ respectively. To show $S_1 + S_2 + S_3$ vanishes, we first
\textbf{claim} that each $S_j$ can be combinatorially reindexed to a
sum over the set
\begin{equation}\label{Ecap}
E_\cap := \{ (e,f) \in E^2 : e \neq f, \ V_e \cap V_f \text{ is nonempty}
\}.
\end{equation}
Notice that $V_e \cap V_f$ is a unique cut-vertex if and only if $(e,f)
\in E_\cap$; denote this vertex by $v_{ef}$.

We now show the claim. The first sum $S_1$ clearly vanishes if $V_e \cap
V_f$ is empty; otherwise since $e \neq f$, it follows that $(e,f) \in
E_\cap$. Moreover, $\Id_{V_e} = \sum_{v \in V_e} {\bf e}_v {\bf e}_v^T$,
so
\[
S_1 = \sum_{(e,f) \in E_\cap} \frac{1}{a_e} {\bf e}_{v_{ef}} {\bf
e}_{v_{ef}}^T M_f.
\]

Similarly, the second sum $S_2$ vanishes unless $v \in f$. Changing the
multiplicative factor $(1 - d_E(v))$ to summing over $\{ e \in E : e \neq
f,\ e \ni v \}$, we have:
\[
S_2 = - \sum_{(e,f) \in E_\cap} \beta_{v_{ef}} {\bf e}_{v_{ef}} {\bf
e}_{v_{ef}}^T M_f.
\]

Finally, the third sum $S_3$ in the right-hand side of~\eqref{Erightside}
involves row vectors of the form ${\bf e}(V(G_{v \to e}))^T M_f$. Notice
that if $V_f \cap V(G_{v \to e})$ is empty then this product row vector
is trivially zero; the same holds if $f \subset G_{v \to e}$, since in
that case
${\bf e}(V(G_{v \to e}))^T M_f = {\bf e}(V_f)^T M_f$,
which vanishes by definition. By the geometry of the graph (i.e., the
definition of $G_{v \to e}$), it follows that if ${\bf e}(V(G_{v \to
e}))^T M_f$ is nonzero, then $(e,f) \in E_\cap$ and $v = v_{ef}$. But
then,
\[
[{\bf e}_v {\bf e}(V(G_{v \to e}))^T]_{V(G_{v \to e})} M_f = {\bf e}_v
{\bf e}_v^T M_f = {\bf e}_{v_{ef}} {\bf e}_{v_{ef}}^T M_f,
\]
from which it follows that
\[
S_3 = \sum_{(e,f) \in E_\cap} \left( \beta_{v_{ef}} - \frac{1}{a_e}
\right) {\bf e}_{v_{ef}} {\bf e}_{v_{ef}}^T M_f.
\]

This shows the above claim. Finally, adding up the previous computations,
\[
C_\mathcal{G} ((D^*_G)^{-1} - \diag(\tauin)) + L_\mathcal{G} = S_1 + S_2
+ S_3 = 0.
\]
This shows~\eqref{Eprojesh24} by Zariski density (see the remarks
after~\eqref{Eprelim}), and concludes the proof.
\end{proof}


To conclude, we remark that the explicit formula for $\dt^{-1}$ for
trees, provided in~\cite[Theorem B]{CK-tree}, is indeed a special case of
Theorem~\ref{Tinverse-general}, by explicitly writing down $(D^*_e)^{-1}$
for every edge (i.e.~strong block) $e$ of the tree.
More generally, one can do the same for hypertrees:

\begin{prop}[Inverse formula for additive-multiplicative
hypertrees]\label{Phypertree-inverse}
Let $G$ be a hypertree, with block-datum $\mathcal{G}$ as in
Proposition~\ref{Phypertree}. Then $\dg^{-1}$ is as in
Theorem~\ref{Tinverse-general} or Proposition~\ref{Ptreelike}, with
\begin{align*}
\alpha_\mathcal{G} = &\ \sum_{e \in E} \frac{-a_e}{1 + {\bf e}(p_e)^T
{\bf d}_e} \left( p_e - 1 + \sum_{v<w \in [p_e]} \frac{(m_{e,v} -
m_{e,w}) (m'_{e,v} - m'_{e,w})}{(1 - m_{e,v} m'_{e,v}) (1 - m_{e,w}
m'_{e,w})} \right),\\
(\tauin)_v = &\ 1 - \sum_{e : v \in e} \frac{1}{1 + {\bf e}(p_e)^T {\bf
d}_e} \sum_{w \in e, w \neq v} \frac{m_{e,w} m'_{e,v} (1 - m_{e,v}
m'_{e,w})}{(1 - m_{e,v} m'_{e,v}) (1 - m_{e,w} m'_{e,w})},\\
(\tauout)_v = &\ 1 - \sum_{e : v \in e} \frac{1}{1 + {\bf e}(p_e)^T {\bf
d}_e} \sum_{w \in e, w \neq v} \frac{m'_{e,w} m_{e,v} (1 - m'_{e,v}
m_{e,w})}{(1 - m_{e,v} m'_{e,v}) (1 - m_{e,w} m'_{e,w})}.
\end{align*}
Moreover, the additive-multiplicative Laplacian matrix is given by
\[
(L_\mathcal{G})_{v,w} = \begin{cases}
0 & \text{if } v \neq w, v \not\sim w,\\
\displaystyle \frac{m_{e,v} m'_{e,w}}{a_e (1 + {\bf e}(p_e)^T {\bf d}_e)
(1 - m_{e,v} m'_{e,v}) (1 - m_{e,w} m'_{e,w})} \qquad & \text{if } v \sim
w \in e,\\
\displaystyle \sum_{e \in E : w \in e} \frac{-1}{a_e (1 + {\bf e}(p_e)^T
{\bf d}_e)} \sum_{u \in e, u \neq w} \frac{m_{e,u} m'_{e,w}}{(1 - m_{e,u}
m'_{e,u}) (1 - m_{e,w} m'_{e,w})} & \text{if } v = w,
\end{cases}
\]
and we also have $C_\mathcal{G}$ given by~\eqref{ECmatrix}, with
\[
\beta_i = \frac{1}{\alpha_\mathcal{G}} \sum_{e \in E : i \in e}
\frac{1}{a_e} \sum_{f \in E : f \subset G_{i \to e}}
\frac{-a_f}{1 + {\bf e}(p_f)^T {\bf d}_f} \left( p_f - 1 + \sum_{v<w \in
[p_f]} \frac{(m_{f,v} - m_{f,w}) (m'_{f,v} - m'_{f,w})}{(1 - m_{f,v}
m'_{f,v}) (1 - m_{f,w} m'_{f,w})} \right).
\]
\end{prop}

For instance, for the weighted $q$-distance matrix whose determinant was
computed in~\cite{S-hypertrees}, one sets $m_{e,v} = m'_{e,v} = \sqrt{q}$
and $a_e = w_e / (q-1)$ for all $e \in E$ and $v \in e$ in the
closed-form expressions above, to obtain $D_q(G)^{-1}$.

Proposition~\ref{Phypertree-inverse} follows from
Lemma~\ref{Lhypertree-mult} and Proposition~\ref{Phypertree} via explicit
computations. In particular, it is possible to suitably modify the
arguments used in proving \cite[Theorem~B]{CK-tree} for trees, and obtain
$\dg^{-1}$ for hypertrees via a different proof. In this case we do not
use the Master GHH identities (i.e.~Theorem~\ref{Tmasterghh}) since all
formulas in the preceding proposition are explicit. A consequence of this
explicit proof is that as for trees~\cite{CK-tree}, one obtains an
alternate, ``computational'' derivation of the closed-form expressions
for $\det(\dg)$ and $\cof(\dg)$ given in Proposition~\ref{Phypertree}.

\subsection{Special case 2: $q$- and additive distance matrices}

Another special case involves the $q$-distance matrix of an arbitrary
graph, in which case $a_e = 1/(q-1)$ for all $e$. In particular,
\[
\dg = D_q(G) = \frac{1}{q-1} (D^*_G - J),
\qquad C_\mathcal{G} = (q-1) \Id
\]
from the definitions and Theorem~\ref{Tmasterghh} (see
Remark~\ref{Rae-constant}). Hence by the Sherman--Morrison formula or via
Theorem~\ref{Tinverse-general},
\[
D_q(G)^{-1} = (q-1) \left[ (D^*_G)^{-1} + \frac{(D^*_G)^{-1} {\bf e}
\cdot {\bf e}^T (D^*_G)^{-1}}{1 - {\bf e}^T (D^*_G)^{-1} {\bf e}}
\right],
\]
and $(D^*_G)^{-1}$ is obtainable from the block submatrices $D^*_{G_e}$
via~\eqref{Eghh-inverse}. In turn, another use of the Sherman--Morrison
formula shows how to convert all occurrences of $(D^*_{G_e})^{-1}$ into
$D_q(G_e)^{-1}, e \in E$. This shows how to obtain a formula for
$\dg^{-1} = D_q(G)^{-1}$ in terms of the $D_q(G_e)^{-1}$; specializing to
$q \to 1$, we also recover an alternate formulation (and proof) of a
recent result of Zhou--Ding--Jia~\cite{ZDJ} -- namely, how to obtain
$D_G^{-1}$ in terms of the $D_{G_e}^{-1}$ for \textit{additive} matrices.
 Note that our method has the added advantage of also obtaining the
corresponding result for the $q$-matrices (for general $q$).
We leave further details to the interested reader.

\subsection*{Acknowledgments}

We thank the American Institute of Mathematics (in their new campus at
CalTech) and
A.K.\ thanks the Institute for Advanced Study, Princeton -- where parts
of this work were carried out -- for their generous hospitality and
excellent working conditions.
P.N.C.\ was partially supported by
INSPIRE Faculty Fellowship research grant DST/INSPIRE/04/2021/002620
(DST, Govt.~of India),
and IIT Gandhinagar Internal Project grant IP/IITGN/MATH/PNC/2223/25.
A.K.\ acknowledges support from
Ramanujan Fellowship SB/S2/RJN-121/2017 and
SwarnaJayanti Fellowship grants SB/SJF/2019-20/14 and DST/SJF/MS/2019/3
from SERB and DST (Govt.~of India),
by a Shanti Swarup Bhatnagar Fellowship from CSIR (Govt.\ of India), and
by the DST FIST program 2021 [TPN--700661].




\end{document}